\documentclass{amsart}
\usepackage{amsmath}
\usepackage{amsfonts}
\usepackage{amsthm, upref}
\usepackage{graphicx}
\usepackage{enumerate}
\usepackage[usenames, dvipsnames]{color}

\input xy
\xyoption{all}

\usepackage{ifthen}
\usepackage{tikz}
\usepackage{appendix}
\usepackage{verbatim}
\usetikzlibrary{decorations.pathmorphing}
\usetikzlibrary{calc}
\usepackage{hyperref}

\renewcommand{\comment}[1]{}
\newcommand{\eq}{\begin{equation}}
\newcommand{\en}{\end{equation}}
\newcommand{\rr}{\mathbb{R}}

\newcommand{\NN}{\mathbb{N}}

\newcommand{\norm}[1]{\left\lVert #1 \right\rVert}
\newcommand{\abs}[1]{\left\lvert #1 \right\rvert}

\newcommand{\mcal}[1]{\mathcal{#1}}
\newcommand{\iprod}[1]{\left\langle #1 \right\rangle }

\newcommand{\tbf}{\textbf}

\renewcommand{\P}{\mathrm{P}}
\newcommand{\E}{\mathrm{E}}
\newcommand{\Var}{\mathrm{Var}}

\newcommand{\simp}{\Delta}

\newcommand{\Hess}{\mathrm{Hess}}
\newcommand{\Diri}{\text{Dirichlet}}

\newcommand{\tmin}{\varrho}

\newcommand{\Beta}{\mathrm{Beta}}

\newcommand{\floor}[1]{\left\lfloor #1\right\rfloor}

\newcommand{\psp}{\mathcal{M}}
\newcommand{\wt}[1]{\widetilde{#1}}
\newcommand{\WF}{\text{WF}}

\begin{document}

\theoremstyle{plain}
\newtheorem{thm}{Theorem}
\newtheorem{lemma}[thm]{Lemma}
\newtheorem{prop}[thm]{Proposition}
\newtheorem{cor}[thm]{Corollary}

\theoremstyle{definition}
\newtheorem{defn}{Definition}
\newtheorem{asmp}{Assumption}
\newtheorem{notn}{Notation}
\newtheorem{prb}{Problem}

\theoremstyle{remark}
\newtheorem{rmk}{Remark}
\newtheorem{exm}{Example}
\newtheorem{clm}{Claim}

\title[High dimensional portfolios]{Exponentially concave functions and high dimensional stochastic portfolio theory}

\author{Soumik Pal}
\address{Department of Mathematics\\ University of Washington\\ Seattle, WA 98195}
\email{soumikpal@gmail.com}



\thanks{This research is partially supported by NSF grant DMS-1308340}

\date{\today}

\begin{abstract} We consider the following problem in stochastic portfolio theory. Are there portfolios that are relative arbitrages with respect to the market portfolio over very short periods of time under realistic assumptions? We answer a slightly relaxed question affirmative in the following high dimensional sense, where dimension refers to the number of stocks being traded. Very roughly, suppose that for every dimension we have a continuous semimartingale market such that (i) the vector of market weights in decreasing order has a stationary regularly varying tail with an index between $-1$ and $-1/2$ and (ii) zero is not a limit point of the relative volatilities of the stocks. Then, given a probability $\eta < 1$ arbitrarily close to one, two arbitrarily small $\epsilon, \delta >0$, and an arbitrarily high positive amount $M$, for all high enough dimensions, it is possible to construct a functionally generated portfolio such that, with probability at least $\eta$, its relative value with respect to the market at time $\delta$ is at least $M$, and never goes below $(1-\epsilon)$ during $[0, \delta]$. There are two phase transitions; if the index of the tail is less than $-1$ or larger than $-1/2$. The construction uses properties of regular variation, high-dimensional convex geometry and concentration of measure under Dirichlet distributions. We crucially use the notion of $(K,N)$ convex functions introduced by Erbar, Kuwada, Sturm \cite{EKS} in the context of curvature-dimension conditions and Bochner's inequalities.  
\end{abstract}

\maketitle

\section{Introduction} We start by recalling the usual set-up of stochastic portfolio theory. See the survey \cite{FKsurvey} for more details. Although we will use words such as \textit{portfolio}, \textit{market}, and \textit{stocks} for alluding to their economic interpretation, no knowledge of these quantities are required for understanding the mathematical problem that follows. At a purely mathematical level we will be studying universal properties of stochastic integrals in high dimensions. However, the economic terminology will give us a real world interpretation of an interesting high-dimensional phenomenon. 

Our state space is the open unit simplex in ${\Bbb R}^n$, $n\ge 2$, defined by 
\eq\label{eq:unitsimp}
\simp^{(n)}:= \left\{  (p_1, p_2, \ldots, p_n):\; p_i > 0 \; \text{for all $i$ and}\; \sum_{i=1}^n p_i=1    \right\}.
\en 
The dimension $n$ represents an equity market that has $n$ stocks. We consider time to be continuous. At any point of time $t$, let $X_i(t) > 0$ be the market capitalization of stock $i$ (i.e., the total dollar amount raised through the shares) at time $t$. Our primary focus is the market weight of stock $i$, defined by
\begin{equation} \label{eqn:marketweight}
\mu^{(n)}_i(t) = \frac{X_i(t)}{X_1(t) + \cdots + X_n(t)}, \quad i = 1, \ldots, n.
\end{equation}
The vector $\mu^{(n)}(t):=\left( \mu_1(t), \ldots, \mu_n(t) \right)$ is called the \textit{market} at time $t$. As time varies we have a function $\left( \mu^{(n)}(t),\; t\ge 0  \right)$ on $\Delta^{(n)}$. We will throughout assume a model under which $\mu^{(n)}(\cdot)$ is a continuous semimartingale. 

At this point we will impose an extra structure over the usual set-up. We will assume that there exists a sequence of semimartingales $\left\{ \mu^{(n)}(\cdot),\; n \in \NN  \right\}$. That is, we have a market for every dimension. It is possible to have them on different probability spaces, but for notational convenience we will assume that they are all represented on a single filtered probability space $\left( \psp,\mathcal{F}= \left\{ \mathcal{F}_t,\; t\ge 0 \right\}, \P  \right)$. This can be done without loss of generality. For $n=1$, we have a trivial situation. To avoid constantly make exception to this case, we will let $\NN$ refer to the set $\{2,3,4,\ldots\}$. 
For any $n\in \NN$, we define a portfolio in dimension $n$ to be an $\mathcal{F}$-predictable process $\left(\pi^{(n)}(t),\; t \ge 0\right)$ with state space $\overline{\simp^{(n)}}$, the closure of $\simp^{(n)}$. In finance, this represents a self-financing long-only portfolio. The $i$th coordinate $\pi^{(n)}_i(t)$ of $\pi^{(n)}$ represents the fraction of its current value invested in stock $i$. Here, the value or the wealth process  of the portfolio refers to the growth (or decay) of $\$ 1$ invested in the portfolio at time zero.   

The only class of portfolios we will be interested in is one where we have a function $\Pi^{(n)}:\simp^{(n)} \rightarrow \overline{\simp^{(n)}}$, where the range is the closed unit simplex. In other words, we define the following adapted process $\pi^{(n)}(t)=\Pi^{(n)}\left(\mu^{(n)}(t)  \right)$ by applying the portfolio function on the current market weights. This gives us a portfolio $\left(  \pi^{(n)}(t),\; t\ge 0 \right)$. For example, if we take the identity map $\pi^{(n)}(t)\equiv \mu^{(n)}(t)$, then this portfolio is called the market portfolio. Its value can be thought of as a capitalization-weighted index (such as S\&P 500 or Russell 1000) which represents the performance of the entire market.

What we are interested is the relative value of a portfolio with respect to the market portfolio, defined as the ratio of the value of the portfolio $\pi^{(n)}$ over that of the market portfolio. We will denote this relative value process by $\left(V_n(t),\; t\ge 0\right)$ where $V_n(0)\equiv 1$. This process can be written as an exponential stochastic integral of the following form:
\[
V_n(t)= \mathcal{E}\left(  \int_0^t \sum_{i=1}^n \frac{\pi^{(n)}_i(u)}{\mu^{(n)}_i(u)} d \mu^{(n)}_i(u) \right), 
\]
where $\mathcal{E}(L_t)$ for a continuous semimartingale $L$ is the process $\exp\left( L_t - \iprod{L}_t/2  \right)$.

We now recall the definition of relative arbitrage from \cite[page 113]{FKsurvey}. A portfolio $\pi^{(n)}$ is said to be a relative arbitrage (or, more colloquially, \textit{beats the market}) if there is a time $T_n>0$ and a $q >0$ such that 
\[
\P\left(  V_n(T_n) \ge 1, \; \inf_{0\le t \le T_n} V_n(t) \ge q \right)=1, \quad \P\left( V_n(T_n) > 1  \right)>0. 
\]  
We will now introduce a new definition by relaxing the above requirement. 

\begin{defn}\label{ref:shortterm} (Asymptotic short term relative arbitrage) We say that a sequence $\left( \pi^{(n)},\; n \in \NN \right)$ is an asymptotic short term relative arbitrage (ASTRA) opportunity if there are two sequences of positive numbers $\left(T_n,\; n \in \NN \right)$ and $\left( M_n,\; n\in \NN  \right)$ such that
\begin{enumerate}[(i)]
\item $\lim_{n\rightarrow \infty} T_n=0$ and $\lim_{n\rightarrow \infty} M_n=\infty$. 
\item $\exists\; q>0$ such that $\P\left( \inf_{0\le t \le T_n} V_n(t) \ge q \right)=1$ for all $n\in \NN$.
\item Moreover, we have $\lim_{n\rightarrow \infty} \P\left( V_n(T_n) \ge M_n \right)=1$. 
\end{enumerate}
\end{defn}

The definition of ASTRA is similar to that of asymptotic arbitrage introduced in \cite{KK94} and studied further in \cite{KS96}, \cite{KK98}, \cite{Kl00}, and \cite{CKT15}.

\bigskip

Our main result shows the existence of such ASTRA opportunities under the following assumptions. A more general version involving regularly varying sequences appear as Theorem \ref{thm:mainprob}. Fix $\alpha\in [1/2,1]$. 
Suppose $\nu^{(n)}\in \Delta^{(n)}$ is given by the coordinates
\[
\nu^{(n)}_i = \frac{i^{-\alpha}}{H^{(\alpha)}_n}, \qquad i=1,2,\ldots, n, \quad \text{where}\quad H^{(\alpha)}_n=\sum_{j=1}^n j^{-\alpha}. 
\]
Plainly, $\nu^{(n)}$ has an asymptotic Pareto or Zipf distribution tail with slope $-\alpha$. For a definition of Pareto or Zipf distribution see \cite{zipfbook}. Let $\Diri(\gamma)$ refer to the Dirichlet probability distribution on $\simp^{(n)}$ with parameter $\gamma \in (0, \infty)^n$.

\begin{asmp}\label{asmp:mainasmp} We assume that there exists a positive sequence $\left(\delta_n,\; n\in \NN \right)$ such that
\[
 \lim_{n\rightarrow \infty} \delta_n=0,\quad \text{but} \quad \lim_{n\rightarrow\infty} \delta_n \sqrt{\log n}=\infty. 
\]
and the following conditions hold. 
\begin{enumerate}[(i)]
\item (\textbf{Initial Dirichlet distribution of norm}) Consider the process $Y(t):=\sqrt{n}\norm{\mu^{(n)}(t) - \nu^{(n)}}$. We assume that $Y(0)$ has the same distribution as when $\mu^{(n)}(0)\sim \Diri\left( n\nu^{(n)}  \right)$. Since the labeling of coordinates is arbitrary, the above statement demands some labeling of coordinate for which it is true.

\item (\textbf{Escape times are not too fast}) Define $M_0:= \sqrt{n}\E \norm{\mu^{(n)} - \nu^{(n)}}$, where $\mu^{(n)}\sim \Diri\left( n\nu^{(n)}  \right)$. It has been proved later that $M_0 < 1$. We assume that for all $1< b_1 < b_2 < \pi/2$, the probability 
\[
\begin{split}
q_n(b_1, b_2)&:=\P\left( \sup_{0\le t \le \delta_n}  Y(t) > b_2 \mid  Y(0) \le b_1\right) \; \text{satisfies}\\
\lim_{n\rightarrow \infty} & q_n(b_1, b_2)=0.  
\end{split}
\]

\item (\textbf{Volatility bounded away from zero}) We will assume that there is a positive constant $\tmin$ such that, for all $n=2,3,\ldots$ and all $1\le i \le n$, we have
\eq\label{eq:volbnd}
\int_0^t \sum_{i=1}^n d\iprod{\mu_i^{(n)}(s), \mu_i^{(n)}(s)} \ge \tmin \int_0^t \left[  \sum_{i=1}^n \left( \mu_i^{(n)}(s) \right)^2 \right]ds, \quad 0\le t \le \delta_n. 
\en 
Here $\iprod{\cdot, \cdot}$ refers to the mutual or quadratic variations of two semimartingales. 
\end{enumerate}
\end{asmp}

In Section \ref{sec:analysis} we show that the stationary Wright-Fisher diffusion model with invariant distribution $\Diri\left( n\nu^{(n)} \right)$ or its deterministic time changes that runs times more slowly satisfy all the assumptions above. In stochastic portfolio theory, the Wright-Fisher model appears as the law of the market weights under the volatility-stabilized market models. See \cite{FK05, PalVSM}. 

The statements of our main results are slightly different depending on whether $\alpha \in [1/2,1)$ or whether $\alpha=1$. We will call the former case as the subcritical case, and the latter as the critical case. Define 
\[
R_n = {H_n^{(2\alpha)}}/{\left(H_n^{(\alpha)}\right)^2}= \frac{\sum_{i=1}^n i^{-2\alpha}}{\left( \sum_{i=1}^n i^{-\alpha}  \right)^2}.
\]

\begin{thm}\label{thm:mainproblimited} Suppose we are given an $\epsilon \in (0,1)$. There exists a sequence of explicit portfolios $\left(\pi^{(n)},\; n \in \NN\right)$ and numbers $1< b_1 < b_2 <\pi/2$, independent of $\alpha$, such that the following hold under Assumption \ref{asmp:mainasmp}. 
\begin{enumerate}[(i)]
\item Almost surely, $\inf_{0\le t \le \delta_n} V_n(t) \ge (1-\epsilon)$ for every $n$.
\item In the subcritical case, there exists a sequence $\left( g_n, \; n \in \NN \right)$ satisfying $g_n \ge c_1 n R_n (\log n)^{-1/2}$, for some positive constant $c_1$ and such that 
\[
\P\left(  V_n(\delta_n) \ge \exp\left( g_n \right)  \right) = 1 - q_n(b_1, b_2) - O\left(\exp\left(-c_0 n^{(1-\alpha)/4}\right)\right),
\]
for some positive constant $c_0$.
\item In the critical case, there exists a sequence $\left( h_n,\; n \in \NN \right)$ which satisfies $h_n \ge c_2 n (\log n)^{-3/2}$ for some positive constant $c_2$, such that 
\[
\P\left(  V_n(\delta_n) \ge \exp\left( h_n \right)  \right) = 1- q_n(b_1, b_2) - O(R_n). 
\]
\end{enumerate}
\end{thm}

It can be easily verified (will be shown later) that $\lim_{n\rightarrow\infty} g_n=\infty$. Thus, $\left( \pi^{(n)},\; n\in \NN  \right)$ is an ASTRA opportunity under our assumptions. Every portfolio $\pi^{(n)}$ is essentially a functionally generated portfolio and therefore explicit. Functionally generated portfolios are given as functions $\pi^{(n)}(t)=\Pi^{(n)}\left( \mu^{(n)}(t) \right)$, where $\Pi^{(n)}$ can be described as a gradient map of the logarithm of a positive concave function on the unit simplex. We recall the notion in Section \ref{sec:KNexp}. 
Also see \cite{PW14} for its connection to optimal transport and the geometry of the unit simplex.

The idea of the construction is the following. In every dimension $n$ we construct a positive concave function around $\nu^{(n)}$ that is \textit{highly concave} in every direction in a certain sense. However, this requires the diameter of the domain of the function to be $O(1/\sqrt{n})$. Although the diameter of the unit simplex is $\sqrt{2}$ in every dimension, its `typical diameter' under $\Diri\left( n\nu^{(n)} \right)$ is about $1/\sqrt{n}$ for large $n$. Thus, such a concave function can be effectively constructed on a large enough subset of $\simp^{(n)}$. By our assumption, the process $\mu^{(n)}$ spends a significant amount of time before its first exit from this large subset with a high probability. These facts, along with the increasing difference between the $\mathbf{L}^2$ and $\mathbf{L}^1$ norms of regularly varying sequences, gives the portfolio a very large relative value.

The proofs break down if either $\alpha < 1/2$ or $\alpha > 1$. The latter case is understandable. The extreme inequality in the market weights effectively reduces the dimensionality of the problem. Since most volatility is generated by the largest market weights, increasing dimension gives us little advantage. The case of $\alpha < 1/2$ is more mysterious. This is the case of too little inequality and is probably related to the fact (see \cite[Theorem 1]{W14}) that the equal-weighted portfolio ($\pi^{(n)}_i(t)\equiv 1/n$) is optimal in a certain sense around a neighborhood of the barycenter $(1/n, 1/n, \ldots, 1/n)$ of the unit simplex $\simp^{(n)}$. It will be interesting to understand this transition better. 

\begin{figure}[t]
\centering
\includegraphics[scale=0.6]{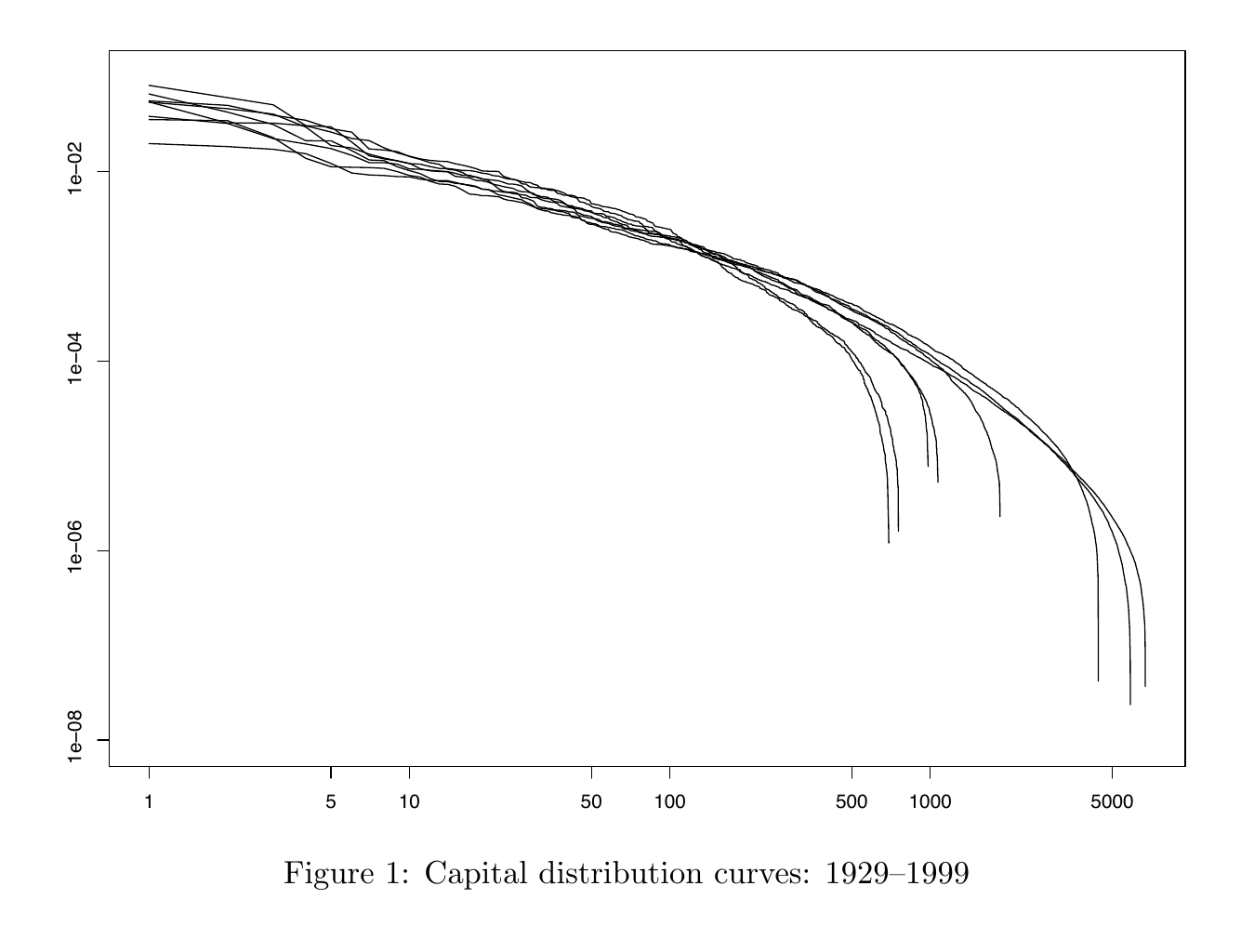}
\caption{Pareto plots of $\log \mu_i$ vs. $\log i$ from 1929--1999 \cite{F02}}
\end{figure}

\subsection{Discussion and comparison with previous work} The existence of relative arbitrage portfolios goes against the \textit{no-arbitrage} theory that has been the fundamental assumption in mathematical economics and finance. However, Robert Fernholz \cite{Fer99}, \cite[Section 3.3]{F02} observed that such portfolios do exist under observable conditions in real-world markets (which martingales do not satisfy). See also \cite{FK05} and \cite{FKK05}.

Fernholz assumed an It\^o process $\mu^{(n)}$ satisfying two conditions. (i) $\exists \;  \eta \in (0,1)$ such that $P\left(\sup_{t\ge 0}\max_{1\le i \le n} \mu^{(n)}_i(t) < 1- \eta\right)=1$. This condition is called \textit{diversity}. And, (ii) a uniform nondegeneracy condition of the diffusion coefficient. This condition is called \textit{being sufficiently volatile}. He then showed that exist explicit portfolio maps $\pi$ such that $\lim_{t\rightarrow \infty} V(t)=\infty$, irrespective of the law of $\mu$. 

The time that it takes to beat the market is of central importance. A natural question to ask is whether, given any small $T >0$, there exists a relative arbitrage portfolio $\pi^{(n)}$ such that $T_n \le T$ with probability one. The challenge, as in any problem of stochastic portfolio theory, is to make as few assumptions (preferably observable) as possible on the market process. Several recent articles have tried to resolve this question. In \cite{FKK05}, it was shown that such portfolios do exist under the condition of diversity. However, the solution is not practical, since to make any significant gain from those portfolios we require an astronomical initial investment. The authors also mention the difficulties of implementation on \cite[page 16]{FKK05}.       

Another solution was posed by \cite{BF08}, inspired by the previous work \cite{FK05} under the condition that the smallest stock in the market has an extreme volatility. To describe their result, define volatility of market weights by $\tau_i(t) = \frac{d}{dt} \iprod{\log \mu^{(n)}_i(t)}$. Then, \cite{BF08} assumes that $\tau_i(t) = 1/\mu^{(n)}_i(t) -1$. In particular, let $m(t)$ refer to the index $i$ such that $\mu^{(n)}_i(t)=\min_j \mu^{(n)}_j$. Since $\mu^{(n)}_{m(t)}(t) \le 1/n$, then $\tau_{m(t)}(t) \ge n-1$ and is growing linearly with dimension $n$. This leads to extreme fluctuation of lower ranked stocks in high dimension which is a drawback of this approach. Moreover, the amount of outperformance by this portfolio is super-exponentially small in the time to beat the market, as remarked on \cite[page 451]{BF08}.   

Recently, another solution has been proposed by Fernholz in \cite{fernshorterm} where it is assumed that the entropy of the market weights satisfy an almost sure time-homogeneity property. As with the previous solutions, the outperformance is not significant enough for these portfolios to be feasible. 

On a complimentary note, Fernholz, Karatzas, and Ruf \cite{FKR15} have shown examples in finite dimensions where short-term relative arbitrages do not exist. 
\medskip

Let us justify our assumptions. Consider condition (i) in Assumption \ref{asmp:mainasmp}. It is well-known that market weights have Pareto tails. See Figure 1 which is taken from \cite{F02}. Each curve represents log of ranked market weights in decreasing order against log of the rank, sampled once every decade from all major U.S. stock markets for eight decades. The curves are roughly linear, and hence the market weights are roughly Pareto, where the qualifier `some labeling' in condition (i) refers to rearrangement according to decreasing size. The appearance of Zipf distribution is not an isolated incident. Pareto tail are ubiquitous in econometric data. See the articles \cite{axtell} and \cite{gabaix}, the book \cite{zipfbook}, or the post \cite{taopost}. In Section \ref{sec:analysis}, we work with real data from June to December of 2015. The Pareto curve with a slope in $[-1/2, -1]$ remains valid as can be seen in Figure \ref{fig:capdist}.

Thus, for any $t$, there exists some labeling of coordinates which makes the sequence $\mu^{(n)}(t)$ approximately Pareto. The factor $n$ in $\Diri\left( n \nu^{(n)} \right)$ reflects the natural scale of fluctuation. One way to think about it is that the uniform distribution over $\simp^{(n)}$ is $\Diri(n\gamma)$ where $\gamma=(1/n, \ldots, 1/n)$. Our $\Diri\left( n \nu^{(n)}\right)$ is the same exponential family as the uniform distribution and keeps the same natural scaling. As will be clear from the proofs, our statements can be generalized to the case of $\Diri\left(c_0n\nu^{(n)}\right)$, for some positive constant $c_0$.

Consider condition (ii). Since Pareto is a stable configuration for the market weights, it cannot escape from a neighborhood of $\nu^{(n)}$ very fast. This is supported by real data as shown in Figure \ref{fig:entropy}.

Condition (iii) is a weaker requirement that the condition that the relative volatilities of the stocks is bounded away from zero. To wit, suppose 
\[
d\iprod{\log \mu^{(n)}_i(t), \log \mu^{(n)}_i(t)} \ge \varrho dt,\quad \text{for all $1\le i \le n$ and $0\le t \le \delta$},
\]
then \eqref{eq:volbnd} follows by It\^o's rule. Hence, our condition is similar to but much weaker than the condition of Fernholz's uniform nondegeneracy condition of the diffusion matrix of $\log \mu^{(n)}$. 

In Section \ref{sec:analysis} we show that the stationary Wright-Fisher diffusion whose invariant distribution is $\Diri\left( n\nu^{(n)} \right)$ satisfy our assumptions. We also justify our assumptions on recent real data (Russell 1000, Jun-Dec 2015) and test the performance of our portfolio for $n=1000$.

\subsection{Outline} Our results are true when the market weights decay as a regularly varying sequence of a range of indices. We recall the notion of regularly varying sequences in Section \ref{sec:rvseq}. In Section \ref{sec:KNexp} we introduce the notion of exponentially concave functions and relate them the recent notion of $(K,N)$ convexity from \cite{EKS}. We also recall the concept of functionally generated portfolios. In Section \ref{sec:Dirichlet}, we prove a concentration of measure theorem for Dirichlet distributions with regularly varying parameters. In Section \ref{sec:seqmarket} we combine the ideas to prove a much more generalized version of Theorem \ref{thm:mainproblimited}. Finally in Section \ref{sec:analysis} we provide theoretical models that satisfy our assumptions and show that the strategy works on real market data.

\section{Preliminaries}\label{sec:rvseq}

\subsection{Notations on orders of magnitude} We will throughout use the following standard notations. The letters $c,c_0, c_1, c_2, \ldots$ and $c',c_0', c_1', c_2', \ldots$ will always refer to positive constants. No two occurrences might refer to the same constant, unless otherwise stated or is obvious. 

For any two sequences of positive numbers $\left( a_n,\; n\in \NN\right)$ and $\left( b_n,\; n\in \NN  \right)$ the following hold. 

\begin{enumerate}[(i)]
\item $a_n = O(b_n)$ means that $a_n \le c_0 b_n$ for some positive constant $c_0$. 
\item $a_n=o(b_n)$ means $\lim_{n\rightarrow \infty} a_n/b_n=0$.
\item $a_n = \Omega(b_n)$ means that $a_n \ge c_0 b_n$ for some positive constant $c_0$.
\item $a_n = \Theta(b_n)$ means $c_0 b_n \le a_n \le c_1 b_n$ for some positive constants $c_0, c_1$. 
\item $a_n \sim b_n$ mean $\lim_{n\rightarrow \infty} a_n/b_n=1$. 
\end{enumerate}

\subsection{Regularly varying sequences}\label{subsec:rvseq}

Consider a non-increasing sequence $a^{(\infty)}=\left( a_1, a_2, \ldots  \right)$ in $(0,1)$. We will denote the $n$th partial sequence by $a^{(n)}:=\left( a_1, a_2, \ldots, a_n  \right)$. Define the usual $\mathbf{L}^p$ norms by 
\[
\norm{a^{(n)}}_p := \left[ \sum_{i=1}^n a_i^p    \right]^{1/p}.
\]
If $p$ is suppressed from the norm notation, $\norm{\cdot}$, then, $p$ is assumed to be $2$. 

We will use the following notations:
\begin{enumerate}[(i)]
\item $H_n:= \sum_{i=1}^n a_i$ for the $\mathbf{L}^1$ norm $\norm{a^{(n)}}_1$. 
\item And, for the ratio of the $\mathbf{L}^2$ norm over the square of the $\mathbf{L}^1$ norm:
\eq\label{eq:whatisrn}
R_n:= \frac{\norm{a^{(n)}}^2_2}{\norm{a^{(n)}}^2_1}= \frac{\sum_{i=1}^n a_i^2}{\left(  \sum_{i=1}^n  a_i \right)^2}. 
\en
\end{enumerate}

The number $R_n$ can be thought of as a measure of inequality among the coordinates of $a^{(n)}$. 
We will work under the following assumption on the sequence $a^{(\infty)}$ that require the coordinates of $a^{(\infty)}$ to decrease fast but not too fast.

\begin{asmp}\label{asmp:primasmp} We assume the following regarding our sequence $a^{(\infty)}$.
\begin{enumerate}[(i)]
\item $\lim_{n\rightarrow \infty} H_n= \infty$.  
\item Moreover, we require
\[
\lim_{n \rightarrow \infty} R_n=0, \quad \text{but}, \quad \lim_{n\rightarrow \infty} n R_n=\infty. 
\]
\end{enumerate}
\end{asmp}

The primary example we will follow throughout this paper is the following. Fix $\alpha \in [0, \infty)$ and consider the sequence $a_i = i^{-\alpha}$. Then $H^{(\alpha)}_n:= \sum_{i=1}^n i^{-\alpha}$. We will call this sequence $\left(H^{(\alpha)}_n,\; n\in \NN \right)$ as the hyperharmonic sequence. The name harmonic sequence will be reserved for the case of $\alpha=1$.

Clearly, condition (i) in Assumption \ref{asmp:primasmp} is satisfied only when $\alpha \in [0,1]$. In that regime, we will throughout employ the following asymptotic estimate:
\eq\label{eq:harmonicest}
H_n^{(\alpha)} \sim \begin{cases}
\log n, & \text{for}\; \alpha=1,\\
n^{1-\alpha}/(1-\alpha), & \text{for}\; \alpha \in [0, 1). 
\end{cases}
\en

Now, $\norm{a^{(n)}}_2^2= \sum_{i=1}^n i^{-2\alpha}= H_n^{(2\alpha)}$. Thus, if $\alpha \in (1/2,1]$, we have 
\eq\label{eq:estimate1}
\lim_{n\rightarrow \infty} \norm{a^{(n)}_2}< \infty, \quad \text{and}, \quad \sqrt{n}\frac{\norm{a^{(n)}}_2}{\norm{a^{(n)}}_1} = 
\begin{cases}
\Theta\left( n^{-1/2 + \alpha}    \right) & \text{for $\alpha \in (1/2, 1)$}\\
\Theta\left(  \sqrt{n}/\log n  \right) & \text{for $\alpha=1$}. 
\end{cases}
\en
Hence, condition (ii) of Assumption \ref{asmp:primasmp} holds as well. In fact, it continues to hold at $\alpha=1/2$ when $\norm{a^{(n)}}_1\sim 2 \sqrt{n}$ and $\norm{a^{(n)}}_2 \sim \sqrt{\log n}$, and hence
\eq\label{eq:estimate2}
R_n= \Theta\left( \frac{\log n}{n}  \right), \qquad n R_n = \Theta\left( \log n  \right).
\en
However, when $\alpha < 1/2$, we get
\[
 \sqrt{n}\frac{\norm{a^{(n)}}_2}{\norm{a^{(n)}}_1} \sim  \frac{1-\alpha}{\sqrt{1-2\alpha}}\frac{n^{1/2} n^{1/2-\alpha} }{n^{1- \alpha}} =  \frac{1-\alpha}{\sqrt{1-2\alpha}} > 0. 
\]
Hence, condition (ii) on Assumption \ref{asmp:primasmp} fails. Thus, Assumption \ref{asmp:primasmp} is valid for the range $\alpha \in [1/2, 1]$ but not elsewhere. 
\bigskip

To generalize the above calculations for other sequences we need to assume more regularity conditions on our sequence $\left( a_n,\; n \in \NN\right)$. We begin by recalling the definition of regularly varying sequences introduced in \cite{svseq}. 

\begin{defn}\label{defn:svseq} We call a sequence $\left(K_n, \; n \in \NN \right)$ of positive terms to be regularly varying with index $\rho\in [0,1]$ if there is another positive sequence $\left( L_n,\; n \in \NN  \right)$ such that 
\begin{enumerate}[(i)]
\item $K_n \sim c L_n$, for some positive constant $c$, 
\item and $\lim_{n\rightarrow \infty} n \left( 1- L_{n-1}/L_n    \right)=\rho$. 
\end{enumerate}
When $\rho=0$ we call the sequence slowly varying. 
\end{defn}

Regularly varying sequences admit a Karamata representation in the spirit of regularly varying functions on the positive half line. The following result is from \cite[Section 2]{svseq}. 

\begin{thm}\label{thm:karamata}
If $\left(  K_n, \; n \in \NN \right)$ is a regularly varying sequence of index $\rho$ then it has the representation
\eq\label{eq:karamata}
K_n = c_n n^\rho \exp\left[  \sum_{j=1}^n \frac{\varepsilon_j}{j}   \right], \quad n \in \NN, 
\en
where $\lim_{n\rightarrow \infty} c_n = c > 0$ and $\lim_{n \rightarrow \infty} \varepsilon_n= 0$. Here $\left(\varepsilon_n, \; n \in \NN\right)$ can be taken to be the sequence $\left(n \left( 1- L_{n-1}/L_n    \right),\; n \in \NN \right)$. 
\end{thm}

\begin{rmk}\label{rmk:normkara} If the sequence $(c_n,\; n \in \NN)$ can be replaced by the constant $c$ above, then the representation is called the normalized Karamata representation. This notion will be relevant later in Definition \ref{defn:zygmund} and Assumption \ref{asmp:primasmp1}.   
\end{rmk}

For $\rho \in [0,1]$, let $\alpha=1-\rho$. The hyperharmonic sequence $\left( H_n^{(\alpha)},\; n\in \NN  \right)$ is a regularly varying sequence of index $\rho$. Here, we take $H_n^{(\alpha)}=K_n$. The fact that it satisfies the definition is the consequence of the estimate \eqref{eq:harmonicest}. Notice that, for $\rho\in [0,1]$, we have
\[
\lim_{n\rightarrow \infty} n\left( 1 - \frac{H_{n-1}^{(\alpha)}}{H_n^{(\alpha)}}  \right)= \lim_{n\rightarrow \infty} \frac{n^{1-\alpha}}{H_n^{(\alpha)}}=(1-\alpha)=\rho.
\]
In that case, we can take $L_n=H_n^{(\alpha)}$ and representation \eqref{eq:karamata} is valid for the choice (say) $\varepsilon_j= j^\rho/H^{(1-\rho)}_j$, $j \in \NN$.

\begin{lemma}\label{lem:rnest}
Suppose $\left( H_n,\; n \in \NN  \right)$ is a regularly varying sequence with index $\rho > 0$, then it satisfies condition (i) in Assumption \ref{asmp:primasmp}. Additionally, if $\rho \in (0,1/2)$ and $\lim_{n\rightarrow \infty} n\left( 1 - H_{n-1}/ H_n   \right)=\rho$, then, it satisfies the conditions in Assumption \ref{asmp:primasmp}. 
\end{lemma}

\begin{proof} Condition (i) follows from the result in \cite{svseq} that regularly varying sequences can be embedded in regularly varying functions by defining $K(x)=K(\floor{x})$ on $[1, \infty)$. However, regular varying functions with a positive index diverge at infinity. 

For condition (ii) we note that $a_i^2 \le a_i$ for every $i$. Hence $R_n \le 1/H_n= o(1)$. On the other hand
\[
\frac{a_n}{H_n} = \left( 1 - \frac{H_{n-1}}{H_n}  \right).
\] 
Hence, by our assumption of regular variation, $a_n = (1+ o(1))\rho H_n/n$. Therefore, 
\[
nR_n= \frac{n\sum_{i=1}^n a_i^2}{H_n^2} = \Theta\left( \frac{n}{H_n^2} \sum_{i=1}^n \frac{H_i^2}{i^2}   \right) = \Omega\left( \frac{n}{H_n^2}  \right), 
\]
the last bound is due to the fact that $H_i \le i$. When $\rho < 1/2$, it follows from Karamata representation that $\lim_{n\rightarrow \infty} n/H_n^2 = \infty$. 

\end{proof}

The case of slowly varying sequences have to be dealt separately since sequences that admit finite positive limits are also slowly varying and do not satisfy Assumption \ref{asmp:primasmp} (i). Moreover, as we see from \eqref{eq:estimate2}, the case of $\rho=\alpha=1/2$ is also delicate due to the presence of logarithmic terms. For these two boundary cases, we will make appropriate assumptions later in the text. In any case, for the rest of the text we will always assume that the sequence $(a_i,\;i\in \NN)$ is such that $\left(H_n,\; n \in \NN\right)$ is regularly varying of index $\rho \in [0,1/2]$ and satisfies Assumption \ref{asmp:primasmp}.

\section{$(K,N)$ exponentially concave functions}\label{sec:KNexp} We begin with a definition. \footnote{Thanks to Prof. W. Schachermayer for suggesting this apt definition.} 

\begin{defn}\label{defn:expcnv} A real-valued function $\varphi$ on an open convex domain $D$ in $\rr^n$ is said to be exponentially concave if $\exp\left( \varphi  \right)$ is a concave function on $D$. For usual practices in convex analysis, we will assume our functions to take the value $-\infty$ outside their domains, thereby extending them to the entire space. 
\end{defn}

As shown in \cite{PW14}, the gradient maps of the above class of functions arise as solutions to a novel and interesting optimal transport problem. Independently, the closely related concept of $(K,N)$ convexity has been very recently introduced in an article by Erbar, Kuwada, and Sturm \cite{EKS}. It is related to the curvature-dimension condition ($K$ standing for a lower bound on the Ricci curvature while $N$ is an upper bound on the effective dimension) and leads to fascinating behavior of entropy over heat flows. For our purpose, we will modify the definition slightly.

\begin{defn} For $K\in \rr$ and $N >0$, a function $\varphi:\rr^n \rightarrow [-\infty, \infty)$ is said to be $(K,N)$ exponentially concave if $\Phi:=\exp\left( N^{-1} \varphi  \right)$ is a concave function on $\rr^n$ satisfying 
\[
\frac{1}{\Phi}\Hess\; \Phi \le -\frac{K}{N}.  
\]
The right hand side above represents the scalar $-K/N$ multiplied with the identity matrix and the inequality is in the sense of two nonpositive-definite matrices. If the Hessian does not exist in the classical sense, interpret it as a measure in the sense of Alexandrov. 

Alternatively, one can write the above as 
\[
\Hess\; \varphi + \frac{1}{N} \left(\nabla \varphi\right) \left(\nabla \varphi\right)' \le - K. 
\]
The above inequality is the usual ordering of two nonpositive-definite matrices.   
\end{defn}

We will focus exclusively on the case when $N=1$ and $K=n$, the number of stocks. Thus, in $\rr^n$, we will only consider $(n,1)$ exponentially concave functions. The following is a fundamental example.

\begin{lemma}\label{lem:expcnv}
For any $x_0 \in \rr^n$, the function $\varphi(x)=\log \cos\left(  \sqrt{n}\norm{x- x_0} \right)$ on the domain $\sqrt{n} \norm{x- x_0}  < \pi/2$ is $(n,1)$ exponentially concave. 

Moreover, consider any $(n,1)$ exponentially concave function $\varphi$ and let $x_0$ be its maximizer. Then, for any $x \in \rr^n$, the following inequality holds:
\[
\varphi(x) \le \varphi(x_0) + \log \cos\left( \sqrt{n}\norm{x-x_0}  \right). 
\]
In particular, the domain of any $(n,1)$ exponentially concave function must be a subset of $\left\{ x:\; \sqrt{n}\norm{x-x_0} < \pi/2  \right\}$.
\end{lemma}

\begin{proof}
For the first claim, it suffices to consider $x_0=0$. Consider any two points $x, y$ and consider the line $\gamma_t= x + t(y-x)$, $t\in [0,1]$, connecting the two. Let $g(t)= \Phi(\gamma_t)$. Then 
\[
\begin{split}
g(t)&= \cos\left( \sqrt{n}\norm{\gamma_t} \right), \quad g'(t)=-\sqrt{n}\sin\left(  \sqrt{n}\norm{\gamma_t}  \right)\frac{\iprod{y-x, \gamma_t}}{\norm{\gamma_t}}\\
g''(t) &= - n  \cos\left( \sqrt{n}\norm{\gamma_t} \right) \left(\frac{\iprod{y-x, \gamma_t}}{\norm{\gamma_t}}\right)^2\\
&- \sqrt{n} \sin\left(  \sqrt{n}\norm{\gamma_t}  \right)\frac{\norm{y-x}^2}{\norm{\gamma_t}} \left[  1 -  \frac{\iprod{y-x, \gamma_t}^2}{\norm{\gamma_t}^2\norm{y-x}^2} \right].
\end{split}
\]

On our domain both $\sin\left(  \sqrt{n}\norm{\gamma_t}  \right)$ and $\cos\left(  \sqrt{n}\norm{\gamma_t}  \right)$ are positive. This observation and the Cauchy-Schwartz inequality implies that $g''(t)\le - n g(t)$. Since this holds for every choice of $x,y$, it proves our claim. 

For the second claim, it is clear that the maximizer exists for this strictly concave function and that $\nabla \varphi(x_0)=0$. 
If $\nabla \varphi$ does not exist at $x_0$, we do a standard approximation by approximating the function with infimal convolutions with a sequence of smooth convex functions. The claim now follows from \cite[Lemma 2.2, part (iii)]{EKS} by considering, as before, the line $\gamma_t= x_0 + t (x- x_0)$. The cited lemma is itself a consequence of comparison theorem for one-dimensional concave functions considered as a function of $t$ over the geodesic $\left(\gamma_t,\; t\in [0,1] \right)$. 

The final statement is now obvious.  
\end{proof}

\begin{rmk}\label{rmk:genmat}
By a change of coordinates it is easy to see that the concept of $(K,N)$ can be generalized in the following way. Let $\Sigma$ be an $n \times n$ nonnegative definite matrix. Define a norm $\norm{x}_{\Sigma}= \sqrt{x' \Sigma x}$. Consider the function $\varphi(x)=\log \cos\left(  \norm{x-x_0}_\Sigma  \right)$ over the domain $ \norm{x-x_0}_\Sigma < \pi/2$. Then, $\Phi= \exp(\varphi)$ satisfies the inequality
\[
\frac{1}{\Phi(x)} \Hess\; \Phi(x) \le - \Sigma,
\]  
and $\varphi$ is the \textit{maximal} one satisfying the above in the sense of Lemma \ref{lem:expcnv}.
\end{rmk}

We now recall the concept of functionally generated portfolios \cite[Chapter 3]{F02}, \cite[Sec 2.3]{PW14}. 

\begin{defn}\label{defn:fgp}
Let $\varphi$ be an exponentially concave function on $\rr^n$ whose domain includes $\simp^{(n)}$. Assume $\varphi$ is differentiable on its domain. We define a map $\pi:\simp^{(n)} \rightarrow \overline{\simp^{(n)}}$ by the following recipe. For every $p\in \simp^{(n)}$, let $v:=\nabla \varphi(p)$ denote the gradient of $\varphi$ at $p$.  Denote the coordinates of $\pi(p)$ by $\left(\pi_1, \ldots, \pi_n \right)$. Then, 
\eq\label{eq:fgp}
\frac{\pi_i}{p_i} = v_i + 1 - \iprod{p,v}. 
\en
In other words, the vector of coordinate ratios $\left(\pi_i/p_i,\; i=1,2,\ldots, n \right)$ is the projection of the gradient $\nabla \varphi(p)$ on the hyperplane $\left\{ y \in \rr^n:\; \iprod{y,p}=1   \right\}$. 
\end{defn}

The above portfolio can now be traded on process of market weights $\mu(\cdot)$. 
Assume that $\varphi$ is twice continuously differentiable on its domain. Let $\Phi=\exp\left( \varphi \right)$. Then the following expression for $\log V(\cdot)$ is known as the Fernholz's decomposition formula. See \cite[Theorem 3.1.5]{F02}.
\eq\label{eq:fernholz}
\log V(t) = \varphi\left( \mu(t) \right) - \varphi\left( \mu(0) \right) - \int_0^t \frac{1}{2\Phi(\mu(s))} \Hess\; \Phi\left( d\mu(s), d\mu(s)  \right). 
\en
Here and throughout, $ \Hess\ \Phi\left( d\mu(s), d\mu(s)  \right)$ refers to the sum of all the elements of the Hadamard product of $\Hess\ \Phi\left( \mu(s) \right)$ and the matrix of infinitesimal mutual variations $\left( d\iprod{\mu_i(s), \mu_j(s)},\; 1\le i,j \le n   \right)$. That is, 
\[
\Hess\; \Phi\left( d\mu(s), d\mu(s)  \right)= \sum_{i=1}^n \sum_{j=1}^n \partial_{ij} \Phi\left( \mu(s) \right) d\iprod{\mu_i, \mu_j}(s),
\] 
and is well-defined when considered as a measure. We will call the following nonnegative nondecreasing process to be the drift process:
\[
\Theta(t):=- \int_0^t \frac{1}{2\Phi(\mu(s))} \Hess\; \Phi\left( d\mu(s), d\mu(s)  \right).
\]

We now define the class of portfolios which we will use to construct our ASTRA sequence. 

\begin{defn}\label{defn:cosine}[Cosine portfolios] Fix $c>0$, $n \in \NN$, a point $x_0 \in \simp^{(n)}$ and consider the portfolio generated by the exponentially concave function $\varphi(x)=\log \cos\left( c\norm{x-x_0} \right)$ as per Definition \ref{defn:fgp}. We will refer to this class of portfolios as the cosine portfolios. 
\end{defn}

\section{The Dirichlet concentration}\label{sec:Dirichlet} Suppose we are given a sequence $a^{(\infty)}$ as in Subsection \ref{subsec:rvseq}. For every $n \ge 2$, consider the unit simplex in $\rr^n$, given in \eqref{eq:unitsimp}. We can get an element in $\Delta^{(n)}$ by defining $\nu^{(n)}= \left( a_1, \ldots, a_n \right)/H_n$. 

Fix $\gamma:=\left(  \gamma_1, \ldots, \gamma_n  \right)\in (0, \infty)^{n}$ for some $n \ge 2$. Recall the Dirichlet distribution with parameter $\gamma$, denoted by $\Diri(\gamma)$. This is a probability distribution on the unit simplex given by the joint density at a point $x\in \Delta^{(n)}$:
\[
\frac{1}{B(\gamma)} \prod\limits_{i=1}^n x_i^{\gamma_i-1}, \quad \text{where}\quad B(\gamma)= \frac{\prod_{i=1}^n \Gamma(\gamma_i)}{\Gamma\left( \sum_{i=1}^n \gamma_i  \right)}.
\] 
Here $\Gamma(\cdot)$ refers to the gamma function.

Consider a sequence $a^{(\infty)}$ satisfying Assumption \ref{asmp:primasmp}. Consider a sequence of Dirichlet distributions with parameters $\gamma^{(n)}:=n\nu^{(n)}$, for $n \ge 2$. This Dirichlet distribution has mean $\nu^{(n)}$ and the same concentration as the uniform distribution over $\Delta^{(n)}$. In fact, if $X\sim \Diri(n \nu^{(n)})$ the following formulas hold
\eq\label{eq:meanvardir}
\E\left( X_i  \right)= \nu_i^{(n)}, \quad \Var\left(  X_i\right)= \frac{\nu_i^{(n)}\left( 1 - \nu_i^{(n)}\right)}{n+1}. 
\en
We start with the following proposition.

\begin{prop}\label{prop:direxpect} Consider the vector $X$ as above. Let $Y_i=\left(X_i - \nu_i^{(n)}\right)^2$ and let $Y= \sum_{i=1}^n Y_i$. Then 
\eq\label{eq:ulbndexpec}
(1- o(1))\sqrt{\frac{2}{\pi} (1 - R_n)} \le \E\left( \sqrt{nY} \right) \le \sqrt{1 - R_n} \le 1. 
\en
\end{prop}

\begin{rmk} Since $\sqrt{2/\pi}\in ( 0.79, 0.8)$, we get $\E(\sqrt{nY})$ is essentially in the interval $[0.79, 1]$ for large enough $n$. 
\end{rmk}

\begin{proof} The upper bound follows from Jensen's inequality:
\[
\E \sqrt{Y} \le \sqrt{\E(Y)} = \sqrt{\sum_{i=1}^n \E Y_i} \le \sqrt{\frac{1}{n+1}\left[\sum_{i=1}^n \nu_i^{(n)}- \sum_{i=1}^n \left(\nu_i^{(n)}\right)^2\right]} \le \sqrt{\frac{1- R_n}{n}}. 
\]
The final inequality is due to the fact that $\nu^{(n)} \in \Delta^{(n)}$. 

For the lower bound, we note that the function $\norm{x}_2$ is convex on $\rr^n$. Thus, again by Jensen's inequality, we get 
\eq\label{eq:expecjensen2}
\E\left(\sqrt{Y} \right) = \E\norm{\abs{X- \nu^{(n)}}}_2 \ge \sqrt{\sum_{i=1}^n \left(  \E \abs{X_i - \nu_i^{(n)}}  \right)^2}.
\en
We recall the following fact about the Beta distribution. Suppose $\xi \sim \Beta(np, n(1-p))$ for $p \in (0,1)$, then (see \cite[Page 14]{betaref}):
\[
\E \abs{\xi- p} = \frac{2 p^{np} (1-p)^{n(1-p)}}{n B(np, n(1-p))}=\frac{2}{n} \frac{p^{np} (1-p)^{n(1-p)} \Gamma(n)}{\Gamma(np) \Gamma(n(1-p))} \sim \sqrt{ \frac{2 p(1-p)}{n \pi}}. 
\] 
The last approximation is due to Stirling's formula. 

Since each $X_i$ is distributed as $\Beta\left( n\nu_i^{(n)}, n\left( 1- \nu_i^{(n)} \right) \right)$, putting the above formula back in \eqref{eq:expecjensen2} we get
\[
\E \left(\sqrt{nY}\right) \ge (1- o(1)) \sqrt{\frac{2}{\pi} \sum_{i=1}^n \nu_i^{(n)}\left( 1 - \nu_i^{(n)} \right)} = (1- o(1))\sqrt{\frac{2}{\pi} (1- R_n)}.  
\]
This completes the proof.
\end{proof}

Proposition \ref{prop:direxpect} shows that the family of random variables $\left(\sqrt{n} \norm{X- \nu^{(n)}},\; n \in \NN \right)$ has bounded mean. We will now prove concentration estimates for the distributions in this sequence around their means.

It is not hard to believe that if $X$ is distributed according to the uniform distribution on the unit simplex, then $\sqrt{n}\norm{X - \E(X)}$ has exponentially decaying tails away from its mean. However, when $\nu^{(n)}$ itself is {atypical} for the uniform distribution this exponential decay starts failing.

We are now ready to prove our concentration estimate. As always, we will assume that $\left(H_n,\; n \in \NN\right)$ is regularly varying of index $\rho \in [0, 1/2]$. However, the proofs are somewhat different for the following two cases: (i) the subcritical case when $\rho \in (0, 1/2]$, and (ii) the critical case when $\rho=0$. We start with (i).

Suppose we are given the sequence $\left( a_n, \; n \in \NN\right)$. We will use the following well-known construction of Dirichlet random variables. Let $Z_1, Z_2, \ldots, Z_n$ be a sequence of independent gamma random variables of scale one such that $\E(Z_i)=n \nu^{(n)}_i$. Let $S_n=Z_1+ \ldots + Z_n$ denote the sequence of their partial sums. Then the vector
\[
\left( \frac{Z_1}{S_n}, \frac{Z_2}{S_n}, \ldots, \frac{Z_n}{S_n}    \right)
\]
is distributed as $\Diri\left( n\nu^{(n)} \right)$. It also follows that, independent of the vector of ratios above, $S_n$ is distributed as Gamma$(n,1)$, a gamma random variable with mean $n$ and scale one. 

We start with by a standard large deviation estimate of $S_n$ around its mean $n$. The result is well-known but we include a short argument anyway. 

\begin{lemma}\label{lem:sconcen}
Fix $u >0$. For all $n> u^2$ we have
\eq\label{eq:boundonS}
\P\left( \abs{S_n-n} > u \sqrt{n}  \right) \le 2 e^{-u^2/4}.
\en
\end{lemma}

\begin{proof}[Proof of Lemma \ref{lem:sconcen}] The convex conjugate of the log-moment generating function of $S_n$ is given by 
\[
I(y)=\begin{cases}
y - n + n \log\left( n/ y \right), &\quad \text{for $y >0$},\\
\infty &\quad \text{otherwise}. 
\end{cases}
\]
For any $\sigma \in (0,1)$ consider $F$ to be the closed interval $[0, (1-\sigma)n] \cup [(1+\sigma)n, \infty)$. Since $I$ is zero at $y=n$, increasing on $[n, \infty)$ and decreasing on $[0,n]$, then 
\[
\begin{split}
\inf_{x\in F} I(x) &=\min\left(  \sigma n + n \log\left(1/(1+\sigma) \right), -\sigma n + n \log\left(1/(1-\sigma)  \right)  \right)\\
&=n \min\left( \sigma - \log(1+\sigma), -\sigma - \log(1-\sigma)  \right)\\  
&\ge \frac{n\sigma^2}{4},\quad \text{since $\log(1 + \sigma) \le \sigma - \sigma^2/4$ for all $-1< \sigma < 1$}. 
\end{split}
\]

We now use Cram\'er's non-asymptotic bound (see Remark (c) on page 27 of \cite{DZ98}). Fix $u > 0$. For any $n$ such that $u < n^{1/2}$, we get 
\eq
\P\left( \abs{S_n-n} > u \sqrt{n}  \right)= \P\left( \abs{S_n-n} > \frac{u}{n^{1/2}} n \right) \le 2 e^{-u^2/4}.
\en
This completes the proof of the lemma. 
\end{proof}

We now consider the subcritical case. Recall $X$ is distributed as $\Diri\left( n\nu^{(n)} \right)$ and the notation $\norm{\cdot}$ refers to the $2$-norm $\norm{\cdot}_2$. Every concentration bound we prove below is for a two sided deviation around the mean, i.e., of the type
\[
\P\left( \abs{ \sqrt{n} \norm{X^{(n)} -\nu^{(n)}} - \E\left(   \sqrt{n} \norm{X^{(n)} -\nu^{(n)}}  \right)} > r \right).
\]
However, for simplicity of notations and keeping our application in mind, we will prove bounds on the following deviations
\[
\P\left( \sqrt{n} \norm{X^{(n)} -\nu^{(n)}} > 1+r \right), \quad \text{and}\quad \P\left( \sqrt{n} \norm{X^{(n)} -\nu^{(n)}} < 1/2-r \right).
\]
This suffices for our purpose since, by Proposition \ref{prop:direxpect}, $\E\left(   \sqrt{n} \norm{X^{(n)} -\nu^{(n)}}\right)$ lies in the interval $[1/2, 1]$ for all large $n$.

\begin{prop}\label{prop:subcritdir}
\tbf{(The subcritical case)} Suppose the sequence $\left( H_n, \; n \in \NN\right)$ is regularly varying with index $\rho \in (0, 1/2]$ and $\lim_{n\rightarrow \infty} na_n/H_n= \rho$. Then
\eq\label{eq:subcritconcen}
\begin{split}
\P&\left( \sqrt{n} \norm{X^{(n)} -\nu^{(n)}} > 1+r \right) = O\left( e^{-c_1 n^{\rho/4}}  \right),
\end{split}
\en
for some choice of positive constant $c_1$ depending on the sequence. 

A similar bound holds, possibly with different constants, for the probability 
\[
\P\left( \sqrt{n} \norm{X^{(n)} -\nu^{(n)}} < 1/2-r \right)=O\left( e^{-c_2 n^{\rho/4}}  \right).
\]
\end{prop}

\begin{proof} Fix $n$. Changing $X^{(n)}$ to the gamma random variables $Z^{(n)}=\left( Z_1, \ldots, Z_n \right)$, we get that for every $n$, 
\eq\label{eq:triangleconcen1}
\begin{split}
\sqrt{n} \norm{X^{(n)} - \nu^{(n)}} &= \sqrt{n}\norm{\frac{Z^{(n)}}{S_n} - \nu^{(n)}}. 
\end{split}
\en

Recall the sequence $\left(R_n,\;n \in \NN\right)$ from \eqref{eq:whatisrn}. Let $\sigma_n= R_n^{-1/4}$. Then 
\[
\lim_{n\rightarrow \infty} \sigma_n=\infty, \qquad \lim_{n\rightarrow \infty} \frac{n}{\sigma^4_n}=\infty.
\] 
Thus, for all large $n$, we get $n > \sigma_n^2$. 

Define the event $E_n:=\left\{  \abs{S_n-n} \le \sigma_n \sqrt{n}   \right\}$. Then, by Lemma \ref{lem:sconcen}, we get 
\eq\label{eq:whatisen}
\P\left( E_n^c \right) \le 2 \exp\left( - \frac{1}{4 \sqrt{R_n}}  \right). 
\en

\comment{
On the event $E_n$, we get
\[
\abs{\frac{n}{S} - 1} \le \frac{\sigma_n \sqrt{n}}{S} \le \frac{\sigma_n \sqrt{n}}{n - \sigma_n \sqrt{n}}. 
\]

Thus, by Jensen's inequality
\eq\label{eq:normerror}
\begin{split}
\E\left[ \frac{1_{E_n}}{\sqrt{n}} \abs{\frac{n}{S} - 1}\norm{Z^{(n)}}  \right] &\le \frac{\sigma_n}{n - \sigma_n \sqrt{n}} \E\norm{Z^{(n)}} \le \frac{\sigma_n}{n - \sigma_n \sqrt{n}} \sqrt{\sum_{i=1}^n \E \left(Z_i^2\right) } \\
& \le  \frac{\sigma_n}{n - \sigma_n \sqrt{n}} \sqrt{\sum_{i=1}^n n\nu_i^{(n)} +\sum_{i=1}^n n^2 \left( \nu_i^{(n)} \right)^2 }\\
& \le  \frac{R_n^{-1/4}}{n - R_n^{-1/4} n} \sqrt{n + n^2 R_n} = O\left( R_n^{-1/4}\right).
\end{split}
\en
}

Define a sequence $k_n:= \floor{\rho^{1/(1-\rho)} n}$, $n\in \NN$. Then 
\[
\lim_{n\rightarrow\infty} k_n=\infty, \quad \lim_{n\rightarrow \infty} \frac{k_n}{n}= \rho^{1/(1-\rho)} \in (0,1). 
\]
We are going to partition the set $\{1,2,\ldots, n\}$ in two parts $A_n:=\{1,2,\ldots, k_n\}$ and $B_n:=\{ k_n+1, \ldots, n\}$. We are going to treat the random variables $\left( Z_i,\; i \in A_n  \right)$ and $\left( Z_i, \; i \in B_n  \right)$ separately. Notice that, by our assumption on regular variation 
\eq\label{eq:correctkn0}
\begin{split}
\lim_{n\rightarrow \infty}\frac{n a_{k_n}}{H_n}&= \lim_{n\rightarrow \infty} \frac{n}{k_n} \lim_{n\rightarrow \infty}\frac{k_n a_{k_n}}{H_{k_n}} \lim_{n\rightarrow \infty} \frac{H_{k_n}}{H_n}= \rho^{-1/(1-\rho)} \rho \lim_{n\rightarrow \infty} \frac{H_{k_n}}{H_n}.
\end{split}
\en

Recall from \cite{svseq}, that $H_n$ can be embedded in a regularly varying sequence $H(x)=H_{\floor{x}}$. It follows that (see \cite[Theorem 1.5.2, page 22]{BinghamBook}) that, for any $b >0$, we have 
\[
\lim_{x\rightarrow \infty} \frac{H(\lambda x)}{H(x)}= \lambda^\rho,
\]
uniformly for $\lambda \in (0, b]$. Hence, by choosing $b=1$, we get $\lim_{n\rightarrow \infty} H_{k_n}/H_n= \rho^{\rho/(1-\rho)}$. Substituting in \eqref{eq:correctkn0}, we get
\eq\label{eq:correctkn}
\lim_{n\rightarrow \infty}\frac{n a_{k_n}}{H_n}= \rho^{1 -1/(1-\rho)} \rho^{\rho/(1-\rho)}=1.
\en
\medskip

We are going to assume, for all $n$ large enough for any $i \in A$, we have $n \nu_i^{(n)} \ge n a_{k_n}/ H_n > 1$, and, for any $i \in B$, we have $n \nu_i^{(n)} \le 1$. Of course, this may not be true at all. But, as will be apparent from the argument, one can repeat the argument by redefining $k_n$ such that $n \nu_i^{(n)} \ge  1+\epsilon$, for $i\in A$, and $n \nu_i^{(n)} \le 1+\epsilon$, for $i \in B$, and then let $\epsilon \rightarrow 0$ to get the same bound.

Fix $n$ large enough for the above conditions to hold. Let $Z^A$ and $Z^B$ denote the vectors $\left( Z_i,\; i \in A  \right)$ and $\left( Z_i,\; i \in B\right)$, respectively. Similarly, partition $\nu^{(n)}$ in two parts, $\nu^A$ and $\nu^B$. 

Now, it is known that a gamma distribution with mean at least one, being a log-concave density, satisfies a Poincar\'e inequality: $\Var(f) \le C \E (f')^2$. If the mean is $\alpha$, the Poincar\'e constant can be taken to be $12\alpha$. See, for example, \cite[Remark 1, page 2716]{BWgamma}. Since every $Z_i$, $i \in A$, has a mean of more than $1$, by the tensorization property of the Poincar\'e inequality, their product measure satisfies a Poincar\'e inequality with a constant 
\[
C_P:=\max_{i \in A} 12 n \nu_i^{(n)}\le \frac{12n}{H_n}. 
\]

The following concentration lemma is a special case of results on general modified log-Sobolev inequalities in \cite{BobLed}. See Theorem 3.1 and Corollary 3.2.

\begin{lemma} Consider an $F:\rr^{k_n}\rightarrow \rr$ that satisfies
\[
\sum_{i\in A} \abs{\partial_i F}^2 \le \alpha^2, \quad \max_{i\in A} \abs{\partial_i F} \le 1.
\]
Then, with respect to the joint distribution of $\left(Z_i,\; i\in A\right)$ the following concentration estimate holds:
\eq\label{eq:concengen1}
\P\left(  F - \E(F) > r   \right) \le \exp\left[ - c_0 \min\left( \frac{r}{\sqrt{C_P}}, \frac{r^2}{\alpha^2 C_P} \right) \right],
\en
for some universal constant $c_0 >0$. 
\end{lemma}

\begin{proof} We follow the statement in \cite[Corollary 4.6, page 62]{Led} which essentially covers our claim. We simply remark on the superficial differences. The i{.}i{.}d{.} structure in the statement of \cite[Corollary 4.6, page 62]{Led} is unimportant since the only thing that is used in the proof is the common Poincar\'e constant. This has been remarked right afterwards by the author. We now need to find an estimate of the constant $K$. The constant $B(\lambda)$ appearing in \cite[Corollary 4.6, page 62]{Led} is bounded by the constant `$3e^5 C/2$' as mentioned in the remark following \cite[Theorem 4.5, page 62]{Led} for $\lambda\le \lambda_0=1/\sqrt{C}$. The rest follows from \cite[Corollary 2.11, page 38]{Led} with a choice of $\lambda_0=1/\sqrt{C_P}$ and $c=3 e^5 C_P/2$. 
\end{proof}
\bigskip

For example, if we apply the above lemma to the function $F(x)=\norm{x}$ on $\rr^{k_n}$ which satisfies $\alpha^2=1$ and $\beta=1$. This gives us
\eq\label{eq:tailbnd1}
P\left(  \norm{Z^A} - \E \norm{Z^A} > r\sqrt{n}  \right) \le \exp\left[ - c_0 \min\left( r \sqrt{H_n}, r^2 H_n  \right)\right].
\en
By shifting the mean, the vector of independent coordinates $\left( Z_i - n \nu_i^{(n)},\; i \in A  \right)$ also satisfies the Poincar\'e inequality with the same constant as above. Hence, we also get, 
\eq\label{eq:tailbnd2}
P\left(  \norm{Z^A - n \nu^A} - \E \norm{Z^A - n\nu^A} > r\sqrt{n}  \right) \le \exp\left[ - c_0 \min\left( r \sqrt{H_n}, r^2 H_n  \right)\right].
\en
This covers the concentration of $\norm{Z^A-\cdot}$. We have to now give a separate argument for $Z^B$ which is not covered by the above arguments. 

For $i \in B$, every $n\nu_i^{(n)} \le 1$ (say). Thus every $Z_i$, $i\in B$, is stochastically dominated by an exponentially distributed random variable with mean one. In particular, by the union bound,
\eq\label{eq:unionbnd1}
\P\left(  \cup_{i\in B} \left\{ Z_i >  n^{1/4} \right\}  \right) \le \sum_{i\in B} P\left(  Z_i > n^{1/4} \right) \le n e^{-n^{1/4}}. 
\en
Let $\wt{E}$ be the event that every $Z_i$, for $i\in B$, is at most $n^{1/4}$. Then, one the event $\wt{E}$, we $Z_i=Z_i1_{\{ Z_i \le n^{1/4} \}}$ for all $i \in B$. 

Consider the vector $\left( n^{-1/4} Z_i 1_{\{ Z_i \le n^{1/4} \}},\; i \in B  \right)$. The law of this vector is a product measure on $[0,1]^{\abs{B}}$. The function $\psi(z):=\norm{z - z_0}$, for some fixed $z_0$, is a convex Lipschitz function on this space. We can now use Talagrand's Gaussian concentration inequality for convex, Lipschitz functions on the unit cube (\cite{Tala1}, \cite[Corollary 3.3]{Led}) to claim
\[
\P\left( \left\{  \norm{Z^B - \nu^B } >  \E\left(\norm{Z^B1_{\{ Z^B \le n^{1/4} \}} - \nu^B }\right) + rn^{1/4} \right\} \cap \wt{E} \right) \le e^{-r^2/2}.   
\]
Here $Z^B1_{\{ Z^B \le n^{1/4}\}}$ refers to the vector $\left( Z_i1_{\{Z_i \le n^{1/4} \}},\; i\in B  \right)$.

We will now adjust the expectation above. By triangle inequality, 
\eq\label{eq:adjexp}
\begin{split}
&\abs{\E \norm{Z^B - \nu^B} - \E \norm{Z^B1_{\{Z^B \le n^{1/4}\}} - \nu^B} } \le \E \norm{Z^B - Z^B1_{\{ Z^B \le n^{1/4}\}}}\\
&\le \sqrt{\sum_{i\in B} \E\left( Z_i - Z_i1_{\{ Z_i \le n^{1/4}\}}  \right)^2}, \quad \text{by Jensen's inequality},\\
&= \sqrt{\sum_{i\in B} \E\left( Z_i^2 1_{\{ Z_i > n^{1/4}  \}}  \right)}\le \sqrt{\abs{B} \E\left(Z^2_0 1_{\{Z_0 \le n^{1/4}\}} \right) }, 
\end{split}
\en
where $Z_0$ is distributed as exponential with mean one, and the inequality is due to stochastic domination. However,
\[
\E\left(Z^2_0 1_{\{Z_0 \le n^{1/4}\}} \right)= \int_{n^{1/4}}^\infty z^2e^{-z}dz= O\left( \sqrt{n}e^{-n^{1/4}}  \right).
\]
Using the trivial bound $\abs{B} \le n$, we get from \eqref{eq:adjexp},
\eq\label{eq:adjexp2}
\abs{\E \norm{Z^B - \nu^B} - \E \norm{Z^B1_{\{Z^B \le n^{1/4}\}} - \nu^B} } = O\left( n^{3/4} e^{-n^{1/4}/2}  \right)= o\left( 1 \right).
\en

And, therefore, using \eqref{eq:unionbnd1} and with a little adjustment to the expectation above, 
\eq\label{eq:tailbnd3}
\P\left(  \norm{Z^B - \nu^B } >  \E\left(\norm{Z^B - \nu^B }\right) + r\sqrt{n} \right) \le  2e^{-r^2\sqrt{n}/2} + n e^{-n^{1/4}}, 
\en
for all large $n$.

We now combining the tail bounds in \eqref{eq:tailbnd2} and \eqref{eq:tailbnd3}. By an application of Jensen's inequality we get
\begin{eqnarray*}
\E \norm{Z^A - n\nu^A} &\le& \sqrt{\sum_{i\in A} \Var(Z_i)} = \sqrt{n H_{k_n}/H_n} \\
\E \norm{Z^B - n\nu^B}  &\le& \sqrt{\sum_{i\in B} \Var(Z_i)} =\sqrt{n\left(1 - {H_{k_n}}/{H_n}\right)}. 
\end{eqnarray*}

Let $p_n:=H_{k_n}/H_n$; thus $1-p_n= 1- H_{k_n}/H_n$. By our choice of $k_n$ and regular variation, we get
\eq\label{eq:limitpn}
\lim_{n\rightarrow \infty} p_n= \rho^{\rho/(1-\rho)} \in (0,1). 
\en
Now, by elementary bounds, we obtain
\eq\label{eq:tailbnd7}
\begin{split}
\P&\left( \norm{Z - n\nu^{(n)}} >  (1 + r)\sqrt{n}  \right) \le \P\left( \norm{Z- n \nu}^2 > (1+r)^2n  \right)\\
& = \P\left( \norm{Z^A- \nu^A}^2 + \norm{Z^B - \nu^B }^2 >  (1+r)^2 n p_n + (1+r)^2 n (1-p_n) \right) \\
&\le \P\left(  \norm{Z^A - \nu^A}^2 > (1+r)^2 n p_n    \right) + \P\left( \norm{Z^B - \nu^B}^2 > (1+r)^2 n(1-p_n)   \right)\\
&= \P\left( \norm{Z^A - \nu^A} > (1+r) \sqrt{np_n}    \right) + \P\left(  \norm{Z^B - \nu^B} > (1+r)\sqrt{n (1-p_n)}   \right)\\
&\le \P\left( \norm{Z^A - \nu^A} > \E\norm{Z^A - \nu^A} + r \sqrt{n p_n}  \right)\\
& + \P\left(  \norm{Z^B - \nu^B} > \E\norm{Z^B - \nu^B} + r \sqrt{n (1-p_n)}    \right)\\
&\le \exp\left[ -c_0 \min\left( r p_n \sqrt{H_n}, r^2 p_n^2 H_n  \right) \right] + \exp\left( - r^2(1-p_n) \sqrt{n}/2 \right) + n e^{- n^{1/4}}.
\end{split}
\en

By Karamata representation, $H_n=O(n^{1/2+\varepsilon})$ for any $\varepsilon>0$ and any $\rho \le 1/2$. Thus we can compress the bound on the right side above and write
\eq\label{eq:tailbnd3}
\P\left( \norm{Z - n\nu^{(n)}} >  (1 + r)\sqrt{n}  \right) \le c_1n\exp\left[ -c_0 \min\left( r \rho^{\rho/(1-\rho)} \sqrt{H_n}, r^2 \rho^{2\rho/(1-\rho)} H_n  \right) \right], 
\en
for some positive constants $c_0, c_1$.

Recall the event $E_n:=\{ \abs{S-n} \le \sigma_n \sqrt{n} \}$ defined above \eqref{eq:whatisen}, and the subsequent discussion. We get
\eq\label{eq:tailbnd5}
\begin{split}
\P &\left( \sqrt{n}\norm{X^{(n)} - \nu^{(n)}} > 1 + r   \right) \le P\left( E^c_n\right) +\P\left( \sqrt{n}\norm{X^{(n)} - \nu^{(n)}} > 1 + r ; E_n     \right)\\
&\le  2 \exp\left( -\frac{1}{4 \sqrt{R_n}}  \right) + \P\left(  n \norm{X^{(n)} - \nu^{(n)}}^2 > (1 +r)^2; E_n \right).
\end{split}
\en
Now, on the event $E_n$, the following estimates hold:
\[
\begin{split}
n &\norm{X^{(n)} - \nu^{(n)}}^2 = n \sum_{i=1}^n \left( \frac{Z_i}{S} - \nu^{(n)}_i  \right)^2 = \frac{n}{S^2} \sum_{i=1}^n \left( Z_i - S \nu^{(n)}_i  \right)^2\\
&\le \left(\frac{n}{n- \sigma_n \sqrt{n}} \right)^2 \frac{1}{n} \left[  \sum_{i=1}^n \left( Z_i - n \nu^{(n)}_i\right)^2 + (S-n)^2 \sum_{i=1}^n \left(\nu^{(n)}_i\right)^2    \right] \\
&+ \left(\frac{n}{n- \sigma_n \sqrt{n}} \right)^2 \frac{1}{n} 2 (S-n)\sum_{i=1}^n \left( Z_i - n \nu^{(n)}_i \right)\nu^{(n)}_i\\
&\le \left(  \frac{1}{1- \sigma_n/\sqrt{n}} \right)^2 \left[ \frac{1}{n} \norm{Z- n \nu^{(n)}}^2  + \sigma_n^2 R_n + \frac{2\sigma_n}{\sqrt{n}} \norm{Z- n\nu^{(n)}} \sqrt{R_n} \right]\\
&\le \left(  \frac{1}{1- \sigma_n/\sqrt{n}} \right)^2 \left[ \frac{1}{n} \norm{Z- n \nu^{(n)}}^2  + \sqrt{R_n} + \frac{2\sqrt[4]{R_n}}{\sqrt{n}} \norm{Z- n\nu^{(n)}}  \right].
\end{split}
\]

Note that, by our assumption on the decay of $R_n$, we get 
\[
\left( 1- \frac{\sigma_n}{\sqrt{n}}  \right)= 1- \frac{1}{\sqrt{n} R_n^{1/4}} = 1 - o\left( \frac{1}{n^{1/4}}  \right)\rightarrow 1.
\]
We have already shown that $\norm{Z- n\nu^{(n)}}/\sqrt{n}$ is $O(1)$ with exponentially decaying tail away from its mean. The other two terms are going to zero in probability at the rate of at least $\sqrt[4]{R_n}$. Thus, it is not hard to see that that the same exponential tail holds for $\sqrt{n} \norm{X^{(n)} - \nu^{(n)}}$ as in \eqref{eq:tailbnd3}, possibly with different values of the constants $c_0, c_1$. Thus, 
\eq\label{eq:probnd1}
\begin{split}
\P&\left( \sqrt{n} \norm{X^{(n)} -\nu^{(n)}} > 1+r \right) \le 2 \exp\left( - \frac{1}{4\sqrt{R_n}}  \right) \\
&+ c_1 n \exp\left[ -c_0 \min\left( r \rho^{\rho/(1-\rho)} \sqrt{H_n}, r^2 \rho^{2\rho/(1-\rho)} H_n  \right) \right].
\end{split}
\en

We simplify the above bound by noting that, since $\left(H_n,\; n \in \NN \right)$ is regularly varying with index $\rho >0$, it follows from Karamata representation that $\lim_{n\rightarrow \infty} n^{-\rho'} H_n=\infty$ for any $\rho' < \rho$. On the other hand
\[
R_n = \frac{\sum_{i=1}^n a_i^2}{H_n^2} \le \frac{\sum_{i=1}^n a_i }{H_n^2} = \frac{1}{H_n}. 
\] 
Thus, combining the above two estimates, we get
\[
\frac{1}{\sqrt{R_n}} \ge \sqrt{H_n} = \Omega\left( n^{\rho'/2}  \right).
\]
Therefore, the bound in \eqref{eq:probnd1} can be simplified to 
\[
\begin{split}
\P&\left(  \sqrt{n} \norm{X^{(n)} - \nu^{(n)}} > 1 + r  \right) = c \exp\left( - n^{\rho'/4} \right) \\
& + c' n \exp\left[ - c_0\min\left( r \rho^{\rho/(1-\rho)} n^{\rho'/2}, r^2\rho^{2\rho/(1-\rho)} n^{\rho'}  \right) \right]\\
&= O\left( e^{-c_1 n^{\rho/4}}  \right), \quad \text{say,}
\end{split}
\]
for some positive constant $c_1$.

This completes the proof of the upper bound of the proposition for the subcritical case. 
The proof of the lower bound is similar by simply substituting $\norm{\cdot}$ by $-\norm{\cdot}$.
\end{proof}

We now handle the critical case when $\left( H_n,\; n \in \NN\right)$ is slowly varying. Slowly varying sequences can display a wide range of properties. For examples, such sequences can have a finite limit. Thus, we will impose further regularity conditions. 

Recall that a slowly varying sequence can be embedded in a slowly varying function by defining $H(x)=H_{\floor{x}}$.
The following definition is taken from \cite[page 24]{BinghamBook} and adapted to our purpose.

\begin{defn}\label{defn:zygmund} [Zygmund class] A positive measurable function on $(0, \infty)$ is said to belong to the Zygmund class if, for every $\alpha > 0$, the function $x^\alpha f(x)$ is ultimately increasing and the function $x^{-\alpha}f(x)$ is ultimately decreasing. A slowly varying sequence is said to belong to the Zygmund class if the corresponding function in which it is embedded is in the Zygmund class. 
\end{defn}

Functions in Zygmund class (see \cite[Theorem 1.5.5, page 24]{BinghamBook}) are slowly varying and have the property that its Karamata representation can be written as $f(x)\sim l(x)$ where
\[
l(x)= c \exp\left( \int_{x_0}^x \frac{\varepsilon(u)}{u} du \right).
\]
for some positive $c, x_0$ and some measurable bounded function $\varepsilon$ vanishing at infinity. That is, in the corresponding representation for sequences \eqref{eq:karamata}, one can replace the convergent sequence $(c_n,\; n\in \NN)$ by the positive limit $c$ and get a normalized Karamata representation. The harmonic sequence is an example of a sequence in the Zygmund class as can be easily verified from definition. 
\medskip

We also define super-slow variation from the left inspired by the so-called super-slow varying functions introduced in \cite{ssvseq} for slowly varying functions. We adapt the definition for sequences.   

\begin{defn}\label{defn:ssvseq}
Suppose $\left( l_n,\; n \in \NN  \right)$ is a nondecreasing positive sequence satisfying $\lim_{n\rightarrow \infty} l_n=\infty$. A slowly varying sequence $\left(K_n,\; n \in \NN \right)$ is said to be super slowly varying from the left with respect to $(l_n,\;n \in \NN)$ if 
\[
\lim_{n\rightarrow \infty} \frac{K_{\lfloor n/ l^\delta_n \rfloor}}{ K_n}=1, \quad \text{uniformly for all $\delta\in [0,1]$}.
\]
\end{defn}

The following conditions generalize the harmonic sequence. 

\begin{lemma}\label{lem:supsv}
Suppose $\left(  H_n,\; n \in \NN \right)$ is in the Zygmund class and satisfies 
\begin{enumerate}[(i)]
\item $\lim_{n\rightarrow \infty} \frac{n a_n}{H_n}=0$.
\item Representation \eqref{eq:karamata} holds for $H_n=K_n$, $c_n \equiv c$, and that $\varepsilon_n \sim 1/\log n$. 
\end{enumerate}
 Then, $\left(H_n\right)$ is super-slowly varying with respect to the sequence $l_n:= \log n$, $n \in \NN$. In particular, uniformly in $\delta\in [0,1]$, we have
\[
\log \frac{H_{\lfloor n / l^\delta_n \rfloor}}{ H_n} = \Theta\left( -\frac{\log \log n}{\log n}\right). 
\] 
\end{lemma}

\begin{proof} For simplicity let us ignore the floor notation $\floor{\cdot}$ from below although it will be implicitly assumed. 

Let $\xi_n:=\exp\left( 1/\varepsilon_n \right)$, $n\in \NN$. Then $\log \xi_n \sim \log n$. By the Karamata representation and our assumptions, we get
\[
\frac{H_{n/ l^\delta_n}}{H_n} = \exp\left[ - \sum_{j=n/l^\delta_n+1}^{n} \frac{1}{j \log \xi_j} \right]. 
\]
Consider the function $\xi:[1, \infty) \rightarrow \infty$ given by $\xi(x)=\xi_{\lfloor x \rfloor}$. Then, it follows by monotonicity
\eq\label{eq:2bndssv}
\frac{1}{\log \xi_n} \int_{n/l_n^\delta +1}^{n+1} \frac{dx}{x} \le \sum_{j=n/l^\delta_n+1}^n \frac{1}{j \log \xi_j} \le \int_{n/ l^\delta_n}^{n} \frac{dx}{x \log \xi(x)}. 
\en 

The lower bound in \eqref{eq:2bndssv} is easy: For some positive constant $c_1$, we get
\[
\frac{1}{\log \xi_n} \int_{n/l_n^\delta +1}^{n+1} \frac{dx}{x} \ge  \frac{c_1}{\log n} \left( \log (n+1) - \log \left( n/\l_n^\delta +1 \right) \right) \sim \frac{c_1\delta \log \log n}{\log n}. 
\]

For the upper bound in \eqref{eq:2bndssv} we change variable to $z=(\log x-\log n)/\log l_n$. By applying substitution to this piecewise continuously differentiable function we get 
\[
\begin{split}
\int_{n/ l^\delta_n}^n \frac{dx}{x \log \xi(x)} &= \log l_n \int_{-\delta}^0 \frac{dz}{\log \xi(\exp\left(z \log l_n + \log n\right))} \\
&\le \frac{\delta \log l_n}{\log \xi(n/l_n^\delta)} \le \frac{c_2\delta \log \log n}{\log n - \delta \log \log n} \le \frac{c_2\delta \log \log n}{\log n}, 
\end{split}
\]
for some positive constant $c_2$. 

Since the upper and the lower bounds are asymptotically of the same order, we get 
\[
\log \frac{H_{n/l^\delta_n}}{H_n} = \Theta\left(-\frac{\log \log n}{\log n}\right). 
\]
Since the right side goes to zero, this completes the proof of the first claim. The second estimate now follows easily.  
\end{proof}

\begin{prop}\label{prop:critdir} 
Assume that the sequence $\left( H_n,\; n\in \NN \right)$ is slowly varying and assume that the conditions of Lemma \ref{lem:supsv} hold. Then the following holds 
\eq\label{eq:critdir}
\P\left(  \sqrt{n} \norm{X^{(n)} - \nu^{(n)}} > 1 + r     \right) \le  \frac{c_3 R_n}{r^2}.
\en
for some positive constants $c_3$. A similar bound holds for the probability 
\[
\P\left(  \sqrt{n} \norm{X^{(n)} - \nu^{(n)}} < 1/2 - r     \right).
\] 
\end{prop}

\begin{proof} The first part of the proof of Proposition \ref{prop:subcritdir} remains unchanged. In particular, the definition of $E_n$ and the bound \eqref{eq:whatisen} remains the same.   

The difference starts with the definition of $k_n$. According to the notation in Lemma \ref{lem:supsv}, consider the sequence $k_n:=\floor{n/ l_n}$ where $l_n=\log n$. Hence, $\left( k_n,\; n\in \NN\right)$ is a nondecreasing sequence such that $\lim_{n\rightarrow \infty} k_n=\infty$. We now show that this choice of $k_n$ satisfies a limit corresponding to \eqref{eq:correctkn}. 
\bigskip
 
\begin{lemma}\label{lem:choosekn} 
We claim that $\lim_{n\rightarrow \infty} \frac{n a_{k_n}}{H_n}=1$.
\end{lemma}

\begin{proof}[Proof of Lemma \ref{lem:choosekn}] By definition
\[
\frac{n a_{k_n}}{H_n} = \frac{n}{k_n} \frac{k_n a_{k_n}}{H_{k_n}} \frac{H_{k_n}}{H_n} \sim l_n k_n\left( 1 - e^{-\varepsilon_{k_n}/k_n} \right)\frac{H_{k_n}}{H_n}. 
\]
By Lemma \ref{lem:supsv} and the assumed conditions, the above in limit is equal to the following 
\[
\lim_{n\rightarrow \infty} \frac{n a_{k_n}}{H_n}=\lim_{n\rightarrow \infty} l_n \varepsilon_{k_n} =\lim_{n\rightarrow \infty} \frac{\log n}{\log n - \log\log n} =1.  
\]
This completes the proof. 
\end{proof}

As before, let $p_n:=H_{k_n}/H_n$. Then $\lim_{n\rightarrow \infty} p_n=1$. The previous lemma again gives us a partition of $\{1, 2,\ldots, n\}$ in $A:=\{1,2,\ldots, k_n\}$ and $B:=\left\{k_n+1, \ldots, n\right\}$ such that $\left( Z_i,\; i\in A  \right)$ has a log-concave density and $\left( Z_i,\; i \in B \right)$ is `small'. 
The difference starts again in \eqref{eq:limitpn}, since in this case $\lim_{n\rightarrow \infty} p_n=1$. In order to account for this difference, we are going to modify \eqref{eq:tailbnd1} and \eqref{eq:tailbnd3}. 

We get an identical tail bound for $\left( Z_i,\; i \in A\right)$ as in \eqref{eq:tailbnd1}. For $\left( Z_i,\; i \in B  \right)$ we will forgo the exponential bound and consider a moment bound. To wit, as before in \eqref{eq:tailbnd7}, we get 
\[
\begin{split}
\P&\left( \norm{Z - n\nu^{(n)}} >  (1 + r)\sqrt{n}  \right) \le \P\left( \norm{Z- n \nu^{(n)}}^2 > (1+2 r)n  \right)\\
&\le \P\left( \norm{Z^A - \nu^A}^2 > (1+ r)n  \right) + \P\left( \norm{Z^B - \nu^B}^2 >  r n  \right)\\
&\le \P\left( \norm{Z^A - \nu^A} > \sqrt{(1+r)n}   \right) + \P\left( \norm{Z^B - \nu^B}^2 >  r n  \right)\\
&\le \P\left( \norm{Z^A - \nu^A} > \E\norm{Z^A - \nu^A} + \sqrt{n}\left( \sqrt{1+r} - 1 \right) \right)\\
& + \P\left(  \norm{Z^B - \nu^B}^2 >  rn   \right).
\end{split}
\] 

The first term on the right can be estimated as before. For the second one we apply Cauchy-Schwarz inequality to get
\[
\begin{split}
\P&\left(  \norm{Z^B - \nu^B}^2 >  r n    \right) = \P\left(  \norm{Z^B - \nu^B}^2 - \E \norm{Z^B - \nu^B}^2 >  rn - n(1-p_n)   \right) \\
&= \P\left(  \norm{Z^B - \nu^B}^2 - \E\norm{Z^B - \nu^B}^2 > (r-1+p_n) n    \right)\\
 &\le \frac{1}{ n^2 (r- 1+ p_n)^2} \Var \left( \norm{Z^B - \nu^B}^2 \right). 
\end{split}
\]
We now compute the variance. By independence of the gamma variables, we get
\eq\label{eq:varcomp}
\begin{split}
\Var&\left( \norm{Z^B - \nu^B}^2 \right) =\Var\left( \sum_{i\in B} \left( Z_i - \E Z_i \right)^2  \right)= \sum_{i=k_n+1}^n \Var \left( Z_i - \E Z_i \right)^2\\
&= \sum_{i=k_n+1}^n \E \left( Z_i - \E Z_i \right)^4 - \sum_{i=k_n+1}^n \left( \Var \left( Z_i \right) \right)^2\\
&= \sum_{i=k_n+1}^n \left( 3 n^2 \left(\nu^{(n)}_i\right)^2 + 6 n \nu^{(n)}_i  \right) - \sum_{i=k_n+1}^n n^2 \left(\nu^{(n)}_i\right)^2 \le 2 n^2 R_n + 6n. 
\end{split}
\en

Note that, by our assumption $\lim_{n\rightarrow \infty} nR_n=\infty$. Thus, $n^2 R_n \gg n$. And hence, 
\eq\label{eq:tailbnd8}
\P\left(  \norm{Z^B - \nu^B}^2 >  r n    \right) \le \frac{c'_3 R_n}{ (r-1+p_n)^2} \le \frac{c_3 R_n}{r^2}, 
\en
for some universal positive constants $c_3, c_3'$. 

Let $r'=\sqrt{1+r}-1$. Therefore, by combining all the tail bounds we get
\[
\begin{split}
\P&\left( \sqrt{n} \norm{X^{(n)} -\nu^{(n)}} > 1+r \right) \le 2 \exp\left( -\frac{1}{4\sqrt{R_n}} \right) \\
&+ \exp\left( - c_0 \min\left( r' \sqrt{H_n}, (r')^2  H_n \right) \right) +  \frac{c_3 R_n}{r^2}.
\end{split}
\]
Considering the leading terms on the right and changing the constants as needed, we get the desired bound in the statement.
\end{proof}

\section{Construction of the ASTRA sequence}\label{sec:seqmarket} 
We will now collect all the conditions we needed on our sequence $(a_n,\; n\in \NN)$ in the following list. 

\begin{asmp}\label{asmp:primasmp1}
Recall that we have a non increasing sequence $(a_n,\; n\in \NN)$ such that each $a_i\in (0,1)$. We assume the following conditions on this sequence.
\begin{enumerate}[(i)]
\item The sequence of partial sums $\left( H_n,\; n\in \NN \right)$ is regularly varying of index $\rho\in [0,1/2]$.  
\item Recall $R_n$ from \eqref{eq:whatisrn}. Then 
\[
\lim_{n \rightarrow \infty} R_n=0, \quad \text{but}, \quad  n R_n=\Omega(\log n) \rightarrow \infty. 
\]
\item $\lim_{n\rightarrow \infty} na_n/H_n= \rho$. 
\item Finally, if $\rho=0$, we assume that $(H_n,\; n\in \NN)$ belongs to the Zygmund class (Definition \ref{defn:zygmund}) and in its normalized Karamata representation (see Theorem \ref{thm:karamata} and Remark \ref{rmk:normkara}) one can take $\varepsilon_n \sim 1/\log n$. 
\end{enumerate}
\end{asmp}

As we have argued before, the hyperharmonic sequences all satisfy the above requirements. 
\bigskip

As mentioned in the Introduction, we assume that there is a probability space on which, for every dimension $n\in \{2,3,\ldots\}$, there is a process of market weights $\left( \mu^{(n)}(t),\; t\ge 0  \right)$ that is a continuous semimartingale on the state space $\simp^{(n)}$. 

Let $\left( \mu^{(n)}_1(t), \ldots, \mu^{(n)}_n(t)   \right)$ be the coordinates of $\mu^{(n)}(t)$. We will make three assumptions on the sequence of processes $\left(  \mu^{(n)},\; n \in \NN  \right)$. To introduce these assumptions, choose the two numbers $0 < r_1 < r_2$ from Assumption \ref{asmp:mainasmp}. First assume that $r_2 < \pi/2-1$ and let $b_1=1+r_1$ and $b_2=1+r_2$. Then, $1 < b_1 < b_2 < \pi/2$. Consider two Euclidean balls in $\rr^n$:
\[
B_1:=\left\{x: \sqrt{n}\norm{x- \nu^{(n)}} < b_1  \right\}, \quad B_2:=\left\{ x: \sqrt{n}\norm{x-\nu^{(n)}} < b_2   \right\}. 
\]
Clearly, $B_1 \subset B_2$.

Now consider the cosine portfolio generated by exponentially concave function $\varphi$ from Definition \ref{defn:cosine} with $x_0=\nu^{(n)}$ and $c=1$. Let $D^{(n)}$ denote its domain, i.e., 
\[
D^{(n)}:=\left\{ x\in \rr^n:\; \sqrt{n} \norm{x - \nu^{(n)}} < \pi/2 \right\}.
\]
Then both $B_1$ and $B_2$ are in $D^{(n)}$. We are interested in the intersection $D^{(n)} \cap \simp^{(n)}$.  

If $r_2 > \pi/2-1$, choose a suitable $0< c< 1$ in the cosine portfolio such that $1+r_2 < \pi/(2 c)$ and apply the following argument which is independent of $c$. Hence, for the rest of the argument we will assume that $1 < b_1 < b_2 < \pi/2$.

We recall our general set-up from Assumption \ref{asmp:primasmp}. In particular, recall the sequences $\left(H_n,\; n \in \NN\right)$ and $\left( R_n,\; n \in \NN\right)$.

\begin{lemma}\label{lem:drift size} Consider a semimartingale process $\left( \mu(s),\; s\ge 0  \right)$ satisfying \eqref{eq:volbnd}. Let $\tau$ be a stopping time such that $\left\{ \mu(s),\; \; 0\le s \le \tau   \right\} \subseteq D^{(n)}$. Then, almost surely, we have the following lower bound on the drift
\[
\Theta(t) \ge \frac{\tmin}{4}  \left( n R_n - \frac{\pi^2}{2} \right) t, \quad \text{for all $0\le t \le \tau$}. 
\]
\end{lemma}

\begin{proof}
By the $(n,1)$ exponential concavity of $\varphi$, whenever $\mu(s)\in D^{(n)}$, we have 
\[
-\frac{1}{\Phi\left( \mu(s) \right)}\Hess\; \Phi\left( d\mu(s), d\mu(s)  \right) \ge n I.
\]
Therefore, for $0\le t \le \tau$, we get
\eq\label{eq:integbnd}
\begin{split}
-\int_0^t \frac{\Hess\ \Phi(d\mu(s), d\mu(s))}{2\Phi\left( \mu(s) \right)}ds &\ge \frac{n}{2} \int_0^t \sum_{i=1}^n  d\iprod{\mu_i, \mu_i}(s)\ge \frac{\tmin n}{2} \int_0^t \sum_{i=1}^n \mu_i^2(s)ds. 
\end{split}
\en 
The last inequality is due to \eqref{eq:volbnd}. 

Now let $p$ be any arbitrary point in $D^{(n)}$. By elementary algebra, 
\[
\left(\nu^{(n)}_i\right)^2 \le 2\left( p_i^2 + \left( p_i - \nu^{(n)}_i  \right)^2   \right).
\]
Therefore, summing over $i$ in the above inequality we get
\eq\label{eq:estrn}
\begin{split}
2n\sum_{i=1}^n p_i^2 \ge n \sum_{i=1}^n \left(\nu^{(n)}_i\right)^2 - 2n \norm{p - \nu^{(n)}}^2 \ge n R_n - \frac{\pi^2}{2}. 
\end{split}
\en
Substituting the above lower bound for every $\mu(s)$ in \eqref{eq:integbnd} gives us the statement of the lemma. 
\end{proof}

We can finally write a statement on the existence of short term arbitrage generalizing Theorem \ref{thm:mainproblimited}. 
Recall $\delta_n, q_n$ from Assumption \ref{asmp:mainasmp}.

\begin{thm}\label{thm:mainprob} 
Assume that our sequence $\left(a_n,\; n\in \NN\right)$ satisfy Assumption \ref{asmp:primasmp1}. Suppose that we are given an $\epsilon \in (0,1)$. Then there exists a sequence of portfolios $\left( \pi_n,\; n \in \NN  \right)$ such that the following conclusions are valid.  Let $V_n(t)$ denote the relative value of the portfolio $\pi_n$ at time $t$. As usual, we always assume $V_n(0)=1$. 

In the subcritical case, when $\rho \in (0, 1/2]$, fix $k\in \NN$. Then the following hold. 
\begin{enumerate}[(i)]
\item Almost surely, $\inf_{0\le t \le \delta_n} V_n(t) \ge (1-\epsilon)$ for every $n$. 
\item With probability $1- \left( q_n + O(\exp(-c_1 n^{\rho/4}))  \right)$, we have 
\[
\log V_n(\delta_n) = \Omega\left( n R_n (\log n)^{-1/2} \right).
\]
\end{enumerate}

In the critical case, when $\rho=0$, the following conclusions hold.  
\begin{enumerate}[(i)']
\item Almost surely, $\inf_{0\le t \le T_n} V_n(t) \ge (1-\epsilon)$ for every $n$. 
\item With probability $1- q_n - O(R_n)$, we have $\log V_n(T_n) = \Omega\left( n R_n (\log n)^{-1/2} \right)$. 
\end{enumerate}
\end{thm}

\begin{proof} First consider the case when $\epsilon =1/2$. Choose $1< b_1 < b_2 < \pi/2$ such that $\cos(b_2) \ge 1-\epsilon=1/2$. This is possible since $\cos(1)\approx 0.54 > 0.5$. Consider the neighborhoods $B_1 \subseteq B_2$ of $\nu^{(n)}$ accordingly. Now consider the portfolio generated by the $(n,1)$ exponentially concave function $\varphi$ given in Lemma \ref{lem:expcnv} with $x_0=\nu^{(n)}$. 

Consider the subcritical case. By Proposition \ref{prop:subcritdir} and Assumption \ref{asmp:mainasmp}, with probability at least $q_n + O(\exp(-c_1 n^{\rho/4}))$, we get $\mu^{(n)}(0)\in B_1$ and the process does not exit $B_2$ by time $\delta_n$. On this event, by Lemma \ref{lem:drift size}, by time $\delta_n$ the total drift for this portfolio is $\Omega\left( nR_n \delta_n \right)=\Omega(\sqrt{\log n})$, by Assumption \ref{asmp:primasmp1} and the fact that $\delta_n=\Omega(1/\sqrt{\log n})$. On the other hand, the range of $\varphi$ insider $B_2$ is 
\[
-\log \cos\left(  b_2 \right)\le - \log(1-\epsilon).
\]
Thus, it follows from \eqref{eq:fernholz} that the relative value of this portfolio never drops below $(1-\epsilon)$. 
On the complement of this event, if $\mu$ exits $B_2$ before time $\delta_n$, we convert our portfolio to the market portfolio. The maximum loss in log relative value is still $(1-\epsilon)$. This proves the result for the subcritical case. 
The critical case is similar.

Now, fix any other $\epsilon \in (0,1)$. If $\epsilon > 1/2$, then we are covered by the case of $\epsilon=1/2$. Suppose $\epsilon \le 1- \cos(1)$. Consider the cosine portfolio from Definition \ref{defn:cosine} by fixing a positive constant $c_0 < 1$ and considering the generating function 
\[
\varphi(x)= \log \cos\left(  c_0 \sqrt{n} \norm{x - x_0}  \right), \quad \text{on the domain}\quad \sqrt{n}\norm{x-x_0} \le \frac{\pi}{2 c_0}. 
\] 
By Remark \ref{rmk:genmat}, on the above domain we get
\[
\frac{1}{\Phi(x)} \Hess\; \Phi(x) \le - c_0 I. 
\]  

Now, choose $c_0$ such that $\cos(c_0) > 1-\epsilon$. Since $\cos(\cdot)$ is decreasing on $[0, \pi/2]$ and $1-\epsilon \ge \cos(1)$, this allows us to choose $c_0 < 1$ to satisfy the requirement of the previous paragraph. Now, as before, choose $1< b_1 < b_2 < \pi/2$ such that $\cos(c_0 b_2) > 1-\epsilon$. We can now repeat the above argument for $\epsilon=1/2$ to reach the same conclusion. The constant $c_0$ is absorbed in the big-O notation.

The case of the remaining interval $\epsilon \in (1- \cos(1), 1/2)$ is now covered by the case of $\epsilon=1-\cos(1) < 1/2$.  
\end{proof}

The proof of Theorem \ref{thm:mainproblimited} now follows as a special case of the above result and the estimates \eqref{eq:estimate1} and \eqref{eq:estimate2}.

\section{Theoretical examples and data analysis}\label{sec:analysis} Of course, one might ask if there is any process $\mu^{(n)}$ that satisfies all the conditions in Assumption \ref{asmp:mainasmp}. We are going to show that the stationary Wright-Fisher (WF) model in dimension $n$ with parameters $n\nu^{(n)}$ satisfies all the conditions of the theorem. We only consider the subcritical case of $\rho\in [0, 1/2)$ for simplicity.

\subsection{Theoretical examples} We will refer to the WF process with parameters $n \nu^{(n)}$ by $\WF\left(n\nu^{(n)}\right)$. This is a diffusion process on state space $\simp^{(n)}$ that satisfies the following stochastic differential equation (SDE):
\[
d\mu(t) = b\left(  \mu(t) \right) dt + \sigma\left( \mu(t) \right) d\beta(t),
\]
where
\begin{enumerate}[(i)]
\item $\beta$ is a $n$-dimensional standard Brownian motion. 
\item $b:\simp^{(n)} \rightarrow \rr^n$ is the function given by the vector difference:
\[
b(p)= \frac{n}{2}\left( \nu^{(n)} - p \right), \qquad p \in \simp^{(n)}.
\]
\item $\sigma$ is a map from $\simp^{(n)}$ to the space of $n\times n$ nonnegative definite matrices. If $p\in \simp^{(n)}$, the $(i,j)$th element of the matrix $\sigma(p)$ is given by 
\[
\sigma_{i,j}(p)= \sqrt{p_i}\left( 1\{ i=j \} - \sqrt{p_i p_j} \right), \qquad 1\le i, j \le n. 
\]
\end{enumerate}

Alternatively, the process can be described via its generator acting on twice continuously differentiable functions $f:\simp^{(n)}\rightarrow \rr$:
\eq\label{eq:genwf}
\mcal{A}_n f(p) := \frac{1}{2} \sum_{i=1}^n \sum_{j=1}^n p_i\left( 1\{i=j\} - p_j \right) \frac{\partial^2 f}{\partial p_i \partial p_j}  + \frac{n}{2}\sum_{i=1}^n \left( \nu^{(n)}_i - p_i \right) \frac{\partial f}{\partial p_i}.  
\en

It is known (see \cite{PalVSM, goiathesis}) that the WF model is the process law of the vector of the market weights under a generalization of the volatility-stabilized model introduced in \cite{FK05}. It is also known (see \cite{PalVSM}) that the unique invariant distribution of $\WF\left( n\nu^{(n)} \right)$ is $\Diri\left( n\nu^{(n)} \right)$. Thus, if $\mu(0)\sim \Diri\left( n \nu^{(n)}  \right)$ and $\left(\mu(t),\; t\ge 0 \right)$ evolves according to $\WF\left( n\nu^{(n)}  \right)$ the process remains stationary.

Let us now consider the process $Y(t)=n\norm{\mu(t) -  \nu^{(n)}}^2$, $t\ge 0$. Consider the Euclidean distance function $y(p)= n\sum_{i=1}^n \left( p_i - \nu^{(n)}_i  \right)^2$. We compute $\mathcal{A}_n y(p)$ to get 
\[
\begin{split}
\mathcal{A}_n y(p) &= -n^2\norm{p- \nu^{(n)}}^2 + n\sum_{i=1}^n p_i(1- p_i) = -ny(p) + n \sum_{i=1}^n p_i(1-p_i).
\end{split}
\] 
Notice that we have the following inequality: $-ny(p) \le \mathcal{A}_ny(p) \le n \left( 1 - y(p)  \right)$.

Moreover, for any $p \in \simp^{(n)}$, consider the matrix $\Sigma(p):=\sigma\sigma'(p)$. Then, it is clear from \eqref{eq:genwf} that 
\[
\Sigma(p) = \text{Diag}(p) - p p', 
\]
where $\text{Diag}(p)$ is the diagonal matrix with diagonal vector $p$. Thus, for any $u\in \rr^n$, we get 
\[
u' \Sigma(p) u = \sum_{i=1}^n p_i u_i^2 - \left( \sum_{i=1}^n p_i u_i \right)^2 \le \left(\max_{1\le i \le n} p_i\right)  \norm{u}^2 \le \norm{u}^2.
\]

Thus, by It\^o's rule
\[
dY(t) = \left( -nY(t) + n \sum_{i=1}^n \mu_i(t)\left( 1- \mu_i(t) \right)    \right) dt + dM(t),  
\]
where $M$ is martingale with quadratic variation
\[
\frac{d}{dt} \iprod{M}_t:=4 n^2 \left( \mu(t) - \nu^{(n)}  \right)' \Sigma\left(\mu(t) \right)\left( \mu(t) - \nu^{(n)}  \right) \le 4 n Y(t). 
\]

Since $Y$ never hits zero (follows from the skew-product result in \cite{PalVSM}), almost surely, one can apply It\^o's rule to derive the SDE of $Z(t)=\sqrt{Y}(t)$: 
\eq\label{eq:sdez}
\begin{split}
d Z(t) &= \frac{1}{2Z(t)} dY(t) -\frac{1}{8 Z^{3}(t)} d\iprod{Y}(t)\\
&= \frac{1}{2Z(t)} d M(t) + \frac{n}{2Z(t)}\left[ -Y(t) + \sum_{i=1}^n \mu_i(t) (1- \mu_i(t))  \right] dt -\frac{1}{8 Z^{3}(t)} d\iprod{Y}(t)\\
&= dN(t) + \frac{n}{2}\left[ -Z(t) + \frac{1}{Z(t)}\sum_{i=1}^n \mu_i(t) (1- \mu_i(t))  \right] dt -\frac{1}{2Z(t)} d\iprod{N}(t).
\end{split}
\en
Here $N$ is a local martingale such that 
\[
\iprod{N}_t= \int_0^t \frac{1}{4 Z^2(s)} d \iprod{M}_s \le \int_0^t \frac{4n Y(s)}{4 Y(s)}ds \le nt, \quad \text{for all $t$}. 
\]

By the Dambis-Dubins-Schwarz theorem \cite[page 174]{KSbook}, we get that there is a standard Brownian motion $\beta$ and a time change $\left(A_t,\;  t \ge 0 \right)$ such that $\P(A_t\le t,\; \forall\; t\ge 0)=1$ and $N(t)=\beta(n A_t)$ for all $t\ge 0$.

Recall $b_2 > b_1 > 1$. Choose $1/2 > \epsilon_1 > \epsilon_2 > 0$. Recall that $\varsigma_a$ is the first hitting time of $a$. Consider the instantaneous drift coefficient of the process $Z$ from \eqref{eq:sdez}. Suppose $Z(0) \in [\epsilon_1, b_1]$, then, during the interval $[0, \varsigma_{b_2}\wedge \varsigma_{\epsilon_2}]$, we get that the coefficient of the instantaneous drift must trivially lie in the interval
\[
\left[ -\frac{n}{2}\left(b_2 + \frac{1}{\epsilon_2}\right), \frac{n}{2\epsilon_2} \right]= [-nc_1, nc_2],  
\]
for some positive constants $c_1, c_2$. Thus, we get that for $t\in [0, \varsigma_{b_2} \wedge \varsigma_{\epsilon_2}]$, we get the following estimate
\eq\label{eq:ddscomp}
-n c_1 t + \beta(n A_t)  \le Z(t) - Z(0) \le n c_2 t + \beta(n A_t).  
\en

Now we have assumed that $Z(0) \in \left[ \epsilon_1,  b_1\right]$. Then, the event $\{ \varsigma_{b_2}\wedge \varsigma_{\epsilon_2} \le  1/n^2  \}$ implies that either $A:=\{ \varsigma_{b_2} \le 1/n^2,\; \varsigma_{b_2} \le \varsigma_{\epsilon_2} \}$ or $B:=\{ \varsigma_{\epsilon_2} \le 1/n^2,\; \varsigma_{\epsilon_2} \le \tau_{b_2} \}$ must have happened. However, by comparing with Brownian motions in \eqref{eq:ddscomp} we get 
\[
\begin{split}
\P(A) &\le \P\left( \sup_{0\le t \le n^{-2}} \beta(n A_t) > (b_2-b_1) - \frac{c_2}{n}  \right)\\
&= \P\left( \sup_{0\le t \le n^{-2}} \beta(n t) > (b_2-b_1) - \frac{c_2}{n}  \right)\\
& \le  \P\left( \sup_{0\le t \le n^{-2}} \beta(t) > \frac{1}{\sqrt{n}}(b_2-b_1) - \frac{c_2}{n^{3/2}}  \right)\le C_3\left( n^2 e^{-c_3 n}  \right),
\end{split}
\]
for some positive constants $c_3, C_3 >0$ and for all large enough $n$ such that ${c_2}/{n} \le (b_2-b_1)/2$ (say). 

Similarly,
\[
\begin{split}
\P(B) &\le \P\left( \inf_{0\le t \le n^{-2}} \beta(nA_t) < (-\epsilon_1 + \epsilon_2) + \frac{c_1}{n}  \right)\\
& \le  \P\left( \inf_{0\le t \le n^{-2}} \beta(t) \le -\frac{\epsilon_1}{\sqrt{n}} - \frac{c_1}{n^{3/2}}  \right)\le C_4\left( n^2 e^{-c_4 n}  \right),
\end{split}
\]
for some positive constants $c_4, C_4 >0$ and for all large enough $n$. 

Combining the above two estimates we get 
\eq\label{eq:hittimeest}
\P\left( \varsigma_{b_2}\wedge \varsigma_{\epsilon_2} \le  1/n^2  \mid Z(0) \in [\epsilon_1, b_1]\right) \le \P(A) + \P(B) = O\left( n^2 e^{-c_0 n}  \right),
\en
for some positive constant $c_0$. 

Now, fix $T>0$, and consider the time interval $[0, T]$. It suffices to take $T=1$ and this is what is followed below. Partition the unit interval in size $1/n^2$, i.e., consider the subintervals $I_k:=\left\{  \left[ (k-1)/n^2, k/n^2  \right], \; k=0,1,2, \ldots, n^2-1   \right\}$. The event that $\sup_{0\le t \le T} Y(t) > b_2$ is contained in the union $\bigcup_{k=0}^{n^2-1} \left\{E_k \cup F_k\right\}$, where  
\begin{eqnarray*}
E_k&=& \left\{  Z(k/n^2) \notin \left[ \epsilon_1,  b_1 \right]\right\},\\
F_k &=& \left\{   Z(k/n^2) \in \left[ \epsilon_1,  b_1 \right],\; \text{and}\; \sup_{t \in I_k} Z(t) > b_2   \right\}.
\end{eqnarray*}

Since $\epsilon_1 < 1/2$, we use Proposition \ref{prop:direxpect} and the two sided concentration estimate in Proposition \ref{prop:subcritdir} under the Dirichlet distribution to obtain $\P(E_k)=\P(E_0)=O\left( e^{-c_1 n^{\rho/4}} \right)$. Also, from the estimate in \eqref{eq:hittimeest} we get $\P(F_k)=\P(F_0)= O\left( n^2 e^{-c_0 n}  \right)$. Hence, by the union bound estimate we get 
\eq\label{eq:escape}
\begin{split}
\P\left( \sup_{0\le t \le 1} Y(t) > b_2  \right) &\le n^2 O\left( e^{-c_1 n^{\rho/4}} \right) + n^2 O\left( n^2 e^{-c_0 n}  \right)= O\left(n^4 e^{-c_0 n^{\rho/4}}\right),  
\end{split}
\en
for some positive constant $c_0$. This verifies condition (ii) in Assumption \ref{asmp:mainasmp} for any bounded sequence $\left(\delta_n,\; n \in \NN\right)$. 

Notice that the same is true if we consider a deterministic time-change $\Gamma^{(n)}_t \le t$ and consider the time-changed stationary Wright-Fisher diffusion. 

We now verify condition \eqref{eq:volbnd}. It follows from the SDE that 
\[
\sum_{i=1}^n \frac{d}{ds} \iprod{\mu_i^{(n)}(s), \mu^{(n)}_i(s)} =  \sum_{i=1}^n \mu_i^{(n)}(s)\left( 1- \mu_i^{(n)}(s)  \right)= 1- \sum_{i=1}^n \left(\mu_i^{(n)}(s)\right)^2.
\]
Before it exits the set $B_2$, it follows from \eqref{eq:estrn} that the above is $\approx 1- R_n$ which is much bigger than $R_n \approx \sum_{i=1}^n \left(\mu_i^{(n)}(s)\right)^2$. In fact, we can again time-change by $\Gamma^{(n)}$. As long as $\Gamma^{(n)}_t \gg t R_n$, our assumptions continue to hold.

\bigskip

For the $\WF\left( n\nu^{(n)} \right)$ model if $\rho\in (0, 1/2)$ the following almost sure statement can be made which is akin to the definition of strong relative arbitrage. 

\begin{thm} Suppose that there is a probability space one which the entire sequence of processes $\left( \mu^{(n)},\; n \in \NN  \right)$ can be realized. Assume that $\mu^{(n)}$ follows stationary WF$\left( n\nu^{(n)} \right)$ model as above. Consider the sequence of portfolios $\left(\pi_n, \; n \in \NN\right)$ and their relative values $\left( V_n,\; n\in \NN  \right)$ from Theorem \ref{thm:mainprob}. Then, w.p. one, for any sample point $\omega$, there exists $n(\omega)\in \NN$ such that $V_{m}(1/\log n) > 1$ for all $m \ge n(\omega)$. 
\end{thm}

\begin{proof} Recall $\varsigma=\inf\left\{ t\ge 0:\; \mu^{(n)}\notin B_2  \right\}$. Consider the bound in \eqref{eq:escape} to get $\P\left(  \mu^{(n)}(0)\in B_1,\; \varsigma > 1/\log n \right)= 1 - O(n^{-2})$. Apply Borel-Cantelli to get an almost sure statement. Since $\rho\in (0,1/2)$, we get $\lim_{n \rightarrow \infty} n R_n T_n=\infty$. 
\end{proof}

Of course, the above is not particularly practical since a priori we do not know which dimension to use. However, it is an interesting allusion to the asymptotic arbitrage theory. 

\begin{figure}[t]\centering
\includegraphics[width=3.8in, height=2.7in]{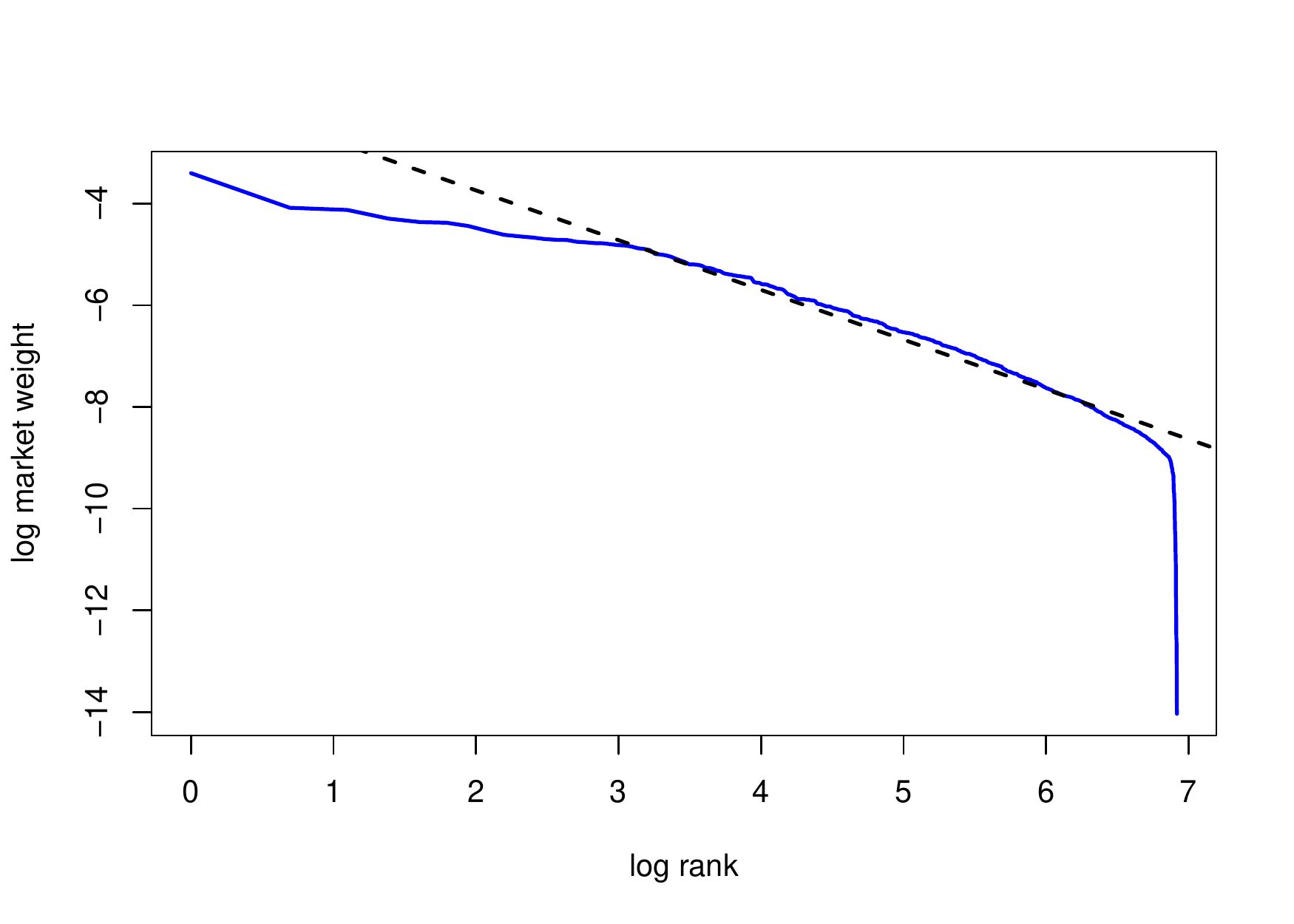}
\caption{Capital distribution curve Jun - Dec 2015  (Source: Russell 1000)}\label{fig:capdist}
\end{figure}

\subsection{Evidence from real data} For our data analysis we consider market capitalization data from the Russell 1000 universe. Russell 1000 is a capitalization-weighted index that constitutes of the largest 1000 companies in the U.S. equity markets. The total market capitalization of all the stocks listed in this index is more than $90\%$ of the entire market capitalization of all the listed U.S. stocks. We consider daily market capitalization data of stocks listed in this index for a period of 130 trading days starting on June 26, 2015, and ending on Dec 30, 2015.

Let us analyze some features of the data to argue that our assumptions are valid. By the nature of the data we can only trade once a day for six months, which is not exactly short term. However, the assumptions do not break down completely. 
In Figure \ref{fig:capdist} we show the capital distribution curve as it appears on the first date, Jun 26, 2015. We have have ranked the market weights and plotted $\log \mu_i$ against $\log i$. The graph shows a linear plot in the log-log scale for the top $700$ stocks and a steep fall for the last $300$. The estimated Pareto slope $\alpha$ for the top part is about $0.95$. If we ignore the bottom $300$ non-Pareto portion, this data lies within the range $[1/2, 1]$ we consider in this paper. Hence, we can take the dimension $n=1000$.

\begin{figure}[thb]
\begin{tabular}{ll}
 \includegraphics[width=2.5in, height=2.2in]{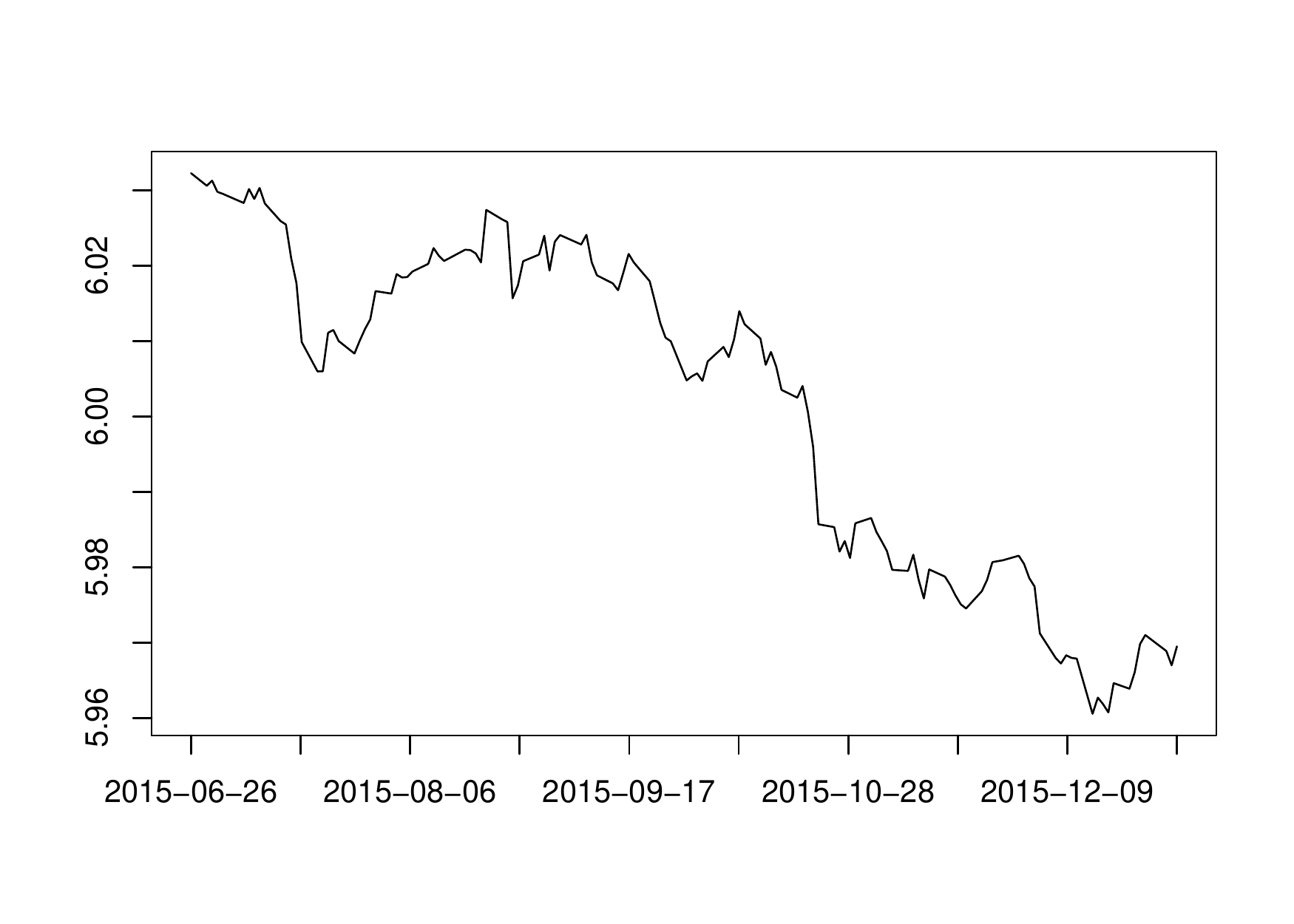} & \includegraphics[width=2.5in, height=2.2in]{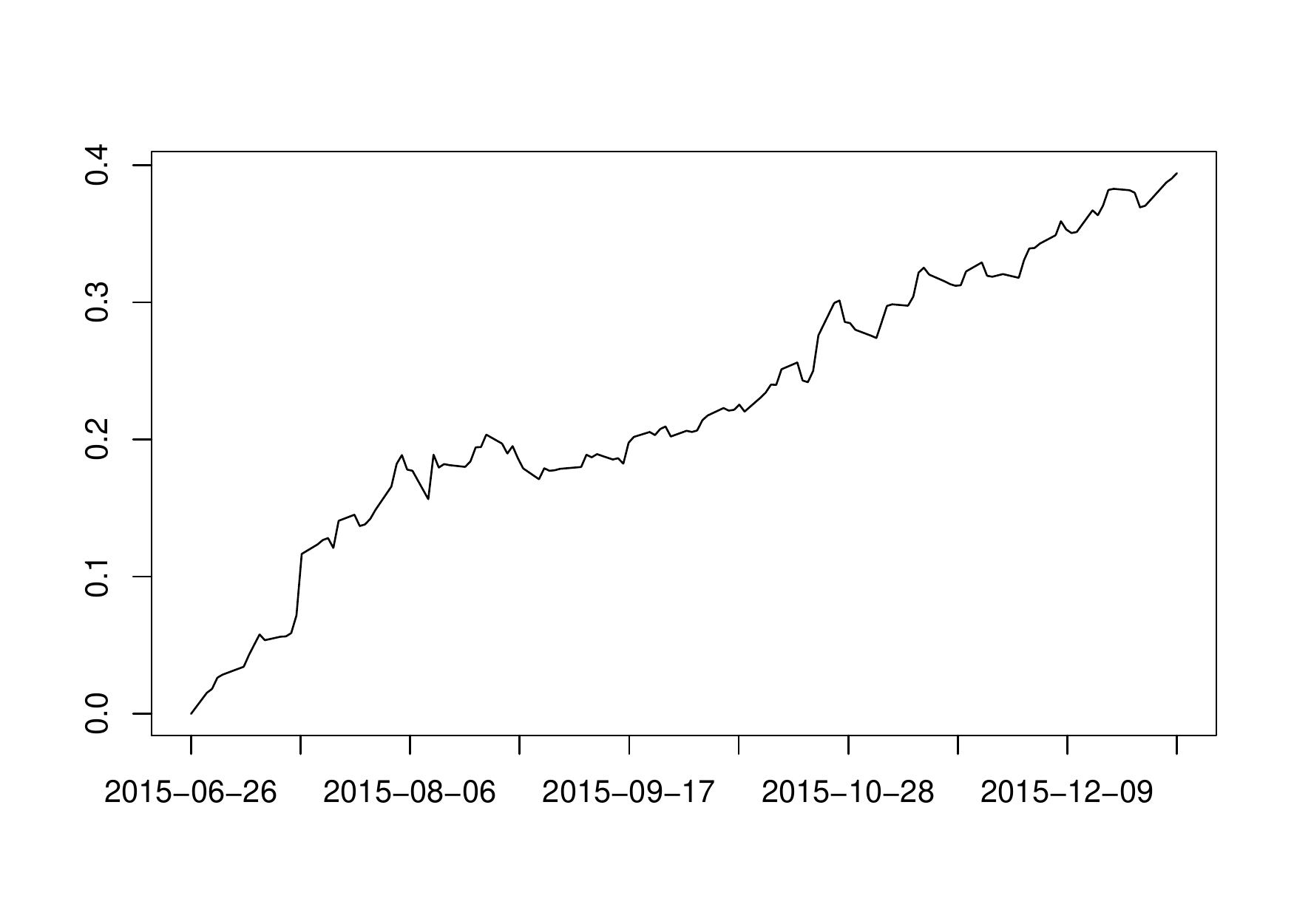}
  \end{tabular}
 \caption{Shannon entropy (left) and scaled Euclidean distance (right) time series.}\label{fig:entropy}
\end{figure}

However, over the course of six months, this slope fluctuates slightly getting closer to one with time. This can be seen by plotting the entropy of the market weights as a time series, as done in the left hand image of Figure \ref{fig:entropy}. The graph show the behavior of the Shannon entropy of the vector of market weights considered as a discrete probability distribution. The entropy decreases with time showing a greater concentration of wealth in the larger stocks and the Pareto slope tending to one. However, our Euclidean distance scaling of $\sqrt{n}$ remains valid as can be seen on the right hand image of Figure \ref{fig:entropy}. The graph shows the scaled Euclidean distance $\sqrt{n}\norm{\mu(t) - \mu(0)}$ against time $t$. The graph shows that this distance varies within $[0.0, 0.4]$ giving credence to the idea that this scaled distance is of order one.

\begin{figure}[htb]\centering
\includegraphics[width=4in, height=3in]{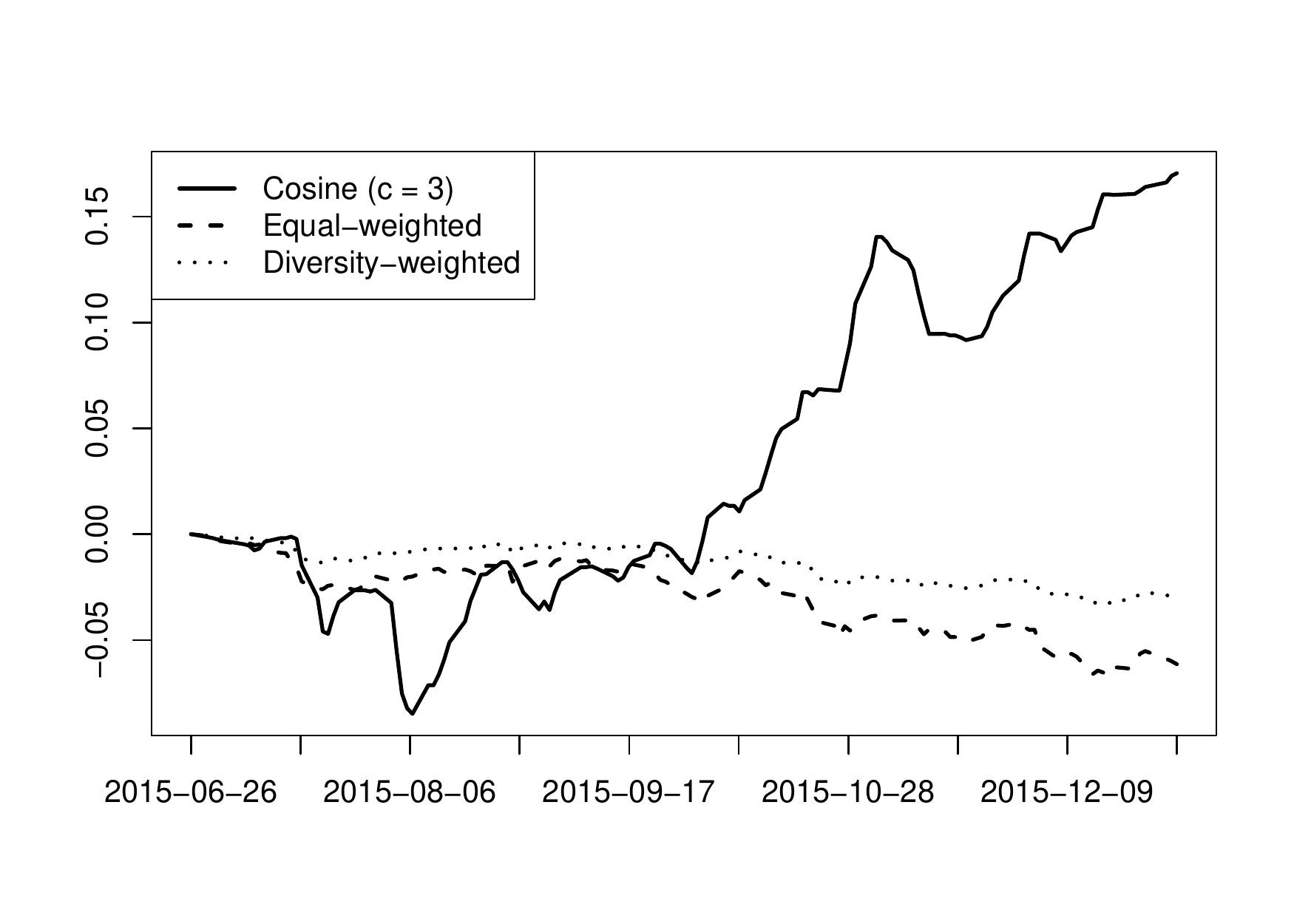}
\caption{Comparison of performances of equal-weighted, diversity-weighted, and the cosine portfolio}\label{fig:performance}
\end{figure}

Thus, we face three problems here: (i) how to reduce this problem to justify short-term? (ii) how to account for a changing slope, and (iii) how to choose the correct constant $c>0$ in the cosine portfolio strategy (Definition \ref{defn:cosine}) or equivalently in the Dirichlet distribution. We resolve these problems by dividing the $130$ days in $13$ periods of successive $10$ days. At the beginning of each period we take the initial market weights to be our $\nu^{(n)}$, which then gets updated in the next period. We choose $c=3$ ad hoc from inspecting the right hand image in Figure \ref{fig:entropy}.

The result of our cosine portfolio (combined over the $13$ periods) for $c=3$ is show in Figure \ref{fig:performance}. The log relative value of the portfolio with respect to the index is plotted on the $y$-axis in the \textbf{bold} line. For comparison, we have also shown the performance over the same data set of the equal weighted portfolio (in dash) and the diversity-weighted portfolio $D_{1/2}$ (in dots). For the definition of the latter see \cite[page 119, eqn. (7.1)]{FKsurvey}. This time period is a particularly bad time for \textit{volatility-harvesting strategies}. This has to do with the decrease in entropy and low volatility in the market. This is the reason why both equal-weighted and the diversity-weighted portfolios underperform the index at the end of six months. The cosine portfolio on the other hand makes significant gains, more than $15\%$ over six months which is an annual rate of $30\%$! But it does underperform initially (although not by much) which gets erased by the gains in the latter half. Better data should allow for a finer understanding of its performance and how to optimize parameters. Unfortunately, the author is limited by the data that he could access.

\section*{Acknowledgement} I am grateful to Jan Maas for pointing out the article \cite{EKS} during my recent visit to Vienna. My thanks to Walter Schachermayer and Matthias Beiglb\"ock for hosting me at Vienna and a lot of very useful discussion. Many thanks to Johannes Ruf and Leonard Wong for numerous comments on a previous draft. Alexander Vervuurt provided the Russell 1000 data. The data analysis was done by Alexander Vervuurt and Leonard Wong. I am very grateful to both of them.

\bibliographystyle{alpha}

\bibliography{infogeo}

\end{document}